\documentclass[10pt,reqno]{article}
\usepackage{latexsym, amsmath, amssymb, a4, epsfig}
\usepackage{graphicx}
\usepackage{color}

\numberwithin{equation}{section}

\newtheorem{thm}{Theorem}[section]

\newtheorem{lem}[thm]{Lemma}

\newtheorem{rem}[thm]{Remark}

\newcommand{\eqnref}[1]{(\ref {#1})}
\newcommand{\ds}{\displaystyle}
\newcommand{\nm}{\noalign{\smallskip}}

\newcommand{\beq}{\begin{equation}}
\newcommand{\eeq}{\end{equation}}

\newcommand{\dist}{\operatorname{dist}}
\newcommand{\diam}{\operatorname{diam}}
\newcommand{\ip}[1]{\left\langle#1\right\rangle}
\newcommand{\supp}{\operatorname{supp}}
\providecommand{\abs}[1]{\lvert#1\rvert}
\providecommand{\norm}[1]{\lVert#1\rVert}

\newcommand{\p}{\partial}

\newcommand{\qed}{\hfill $\Box$ \medskip}
\newcommand{\ie}{{\it i.e.}}
\newcommand{\Fcal}{\mathcal{F}}
\newcommand{\Mcal}{\mathcal{M}}
\newcommand{\Ncal}{\mathcal{N}}
\newcommand{\Rcal}{\mathcal{R}}
\newcommand{\Scal}{\mathcal{S}}
\newcommand{\RR}{\mathbb{R}}

\newcommand{\CC}{\mathbb{C}}

\newcommand\Lt{L^*}
\providecommand{\set}[1]{\{#1\}}
\providecommand{\Biggabs}[1]{\Biggl\lvert#1\Biggr\rvert}
\providecommand{\abs}[1]{\lvert#1\rvert}
\providecommand{\norm}[1]{\lVert#1\rVert}

\newcommand{\Ge}{\epsilon}
\newcommand{\Gg}{\gamma}
\newcommand{\Gm}{\mu}
\newcommand{\Go}{\omega}
\newcommand{\GD}{\Delta}
\newcommand{\GG}{\Gamma}
\newcommand{\GO}{\Omega}

\begin{document}
\title{Sharp estimates for the Neumann functions and applications to quantitative
photo-acoustic imaging in inhomogeneous media\thanks{\footnotesize
This work was supported by the ERC Advanced Grant Project
MULTIMOD--267184 and NRF grants No. 2009-0090250, 2010-0017532, and
2011-0004177.}}

\author{Habib Ammari\thanks{\footnotesize Department of Mathematics and Applications, Ecole Normale Sup\'erieure,
45 Rue d'Ulm, 75005 Paris, France (habib.ammari@ens.fr).} \and
Hyeonbae Kang\thanks{Department of Mathematics, Inha University,
Incheon 402-751, Korea (hbkang@inha.ac.kr).}  \and Seick
Kim\thanks{\footnotesize Department of Mathematics, Yonsei
University, Seoul 120-749,  Korea (kimseick@yonsei.ac.kr).}}

\maketitle

\begin{abstract}
We obtain sharp $L^p$ and H\"older estimates for the Neumann
function of the operator $\nabla \cdot \gamma \nabla - ik$ on a
bounded domain. We also obtain quantitative description of its
singularity. We then apply these estimates  to quantitative
photo-acoustic imaging in inhomogeneous media. The problem is to
reconstruct the optical absorption coefficient of a diametrically
small anomaly from the absorbed energy density.
\end{abstract}

\noindent {\footnotesize {\bf AMS subject classifications.} 31B20,
35J25, 35J08}

\noindent {\footnotesize {\bf Key words.} Neumann function,
variable coefficients, quantitative photo-acoustic imaging,
asymptotic expansion, diffusion approximation}

%
\section{Introduction and motivation}
%

The purpose of this paper is to derive sharp estimates of the Neumann function of the operator
$\nabla \cdot \gamma \nabla - ik$ and its derivatives, where $\gamma$ is an (scalar) elliptic coefficient defined on a bounded domain $\GO \subset \RR^d$ ($d \ge 3$) and $k$ is a positive constant. The Neumann function of $\nabla \cdot \gamma \nabla - ik$ in $\GO$ is the function $N:\Omega\times\Omega\to \CC\cup\set{\infty}$ satisfying
\beq\label{basicprob}
\left\{
\begin{aligned}
- (\nabla \cdot \gamma \nabla - ik) N(\cdot,y) &=\delta_y \quad \text{in }\;\Omega,\\
\gamma \nabla N(\cdot,y)\cdot n &= 0 \quad \text{on }\;\partial\Omega,
\end{aligned}
\right. \eeq for all $y\in\Omega$, where $\delta_y$ is the Dirac
mass at $y$ and $n$ is the outward unit normal vector field on
$\p\GO$ (see subsection \ref{sec:nf} for a precise definition of
the Neumann function). The function $N(x,y)$ has a singularity at
$x=y$. We are particularly interested in describing in a
quantitative manner the singularity of $N(x,y)$ and its dependence
on the parameter $k$.

The investigation of this paper is motivated by quantitative
photo-acoustic imaging, particularly by the recent work
\cite{ABJK2}.



The purpose of quantitative photo-acoustic imaging is to image the
optical absorption coefficient from  the absorbed energy. The
absorbed energy is  obtained from boundary measurements of the
pressure wave induced by the photoacoustic effect. We refer to
\cite{ABJK} and references
 therein for recent development on this inverse problem. Reconstruction of the
 optical absorption coefficient, $\mu_a$,
from the absorbed energy, $A$, is more delicate than the
reconstruction of the absorbed energy from the pressure wave since
$\mu_a$ is related to $A$ in an implicit and non-linear
 way (see Section \ref{sec:app}). One direction of research in quantitative
  photo-acoustic imaging is to reconstruct the absorption coefficient of diametrically small unknown
  anomalies. In \cite{ABJK2, ABJW}, efficient methods
  to reconstruct $\Gm_a$ from $A$ are proposed and implemented numerically when
  there is a small absorbing anomaly in the background medium. The methods use in an essential
   way an asymptotic expansion of $A$ in terms of $\mu_a$ when the diameter of the
   anomaly tends to $0$. The asymptotic expansion is derived using estimates of the
    Neumann function under the assumption that the scattering coefficient of the
    medium is constant. In order to extend the results of \cite{ABJK2, ABJW}
    to inhomogeneous media, we shall derive sharp estimates of the
    Neumann function of problem \eqnref{basicprob}, which is exactly what this
    paper aims at.

To describe the kinds of results obtained in this paper, let us
fix a point $z \in \GO$ ($z$ indicates the location of the
anomaly), and let $\Gg^*:=\Gg(z)$. Let $\GG(x):= -1/(4\pi |x|)$ be
a fundamental solution of the Laplacian in three dimensions. Then,
we will show by precise estimates depending on $k$ that the
singularity of $N(x,z)$ for $x$ near $z$ is of the form
$\frac{1}{\Gg^*} \GG(x-z)$. We also show that the singularity of
the derivatives of $N(x,z)$ is given by the derivatives of
$\frac{1}{\Gg^*} \GG(x-z)$. We also derive $L^p$, pointwise, and
H\"older estimates of the Neumann function $N$. We then use these
estimates to derive an asymptotic expansion in inhomogeneous media
where the scattering coefficient $\mu_s$ is not constant.

This paper is organized as follows. In Section 2, we derive $L^p$
and pointwise estimates of the Neumann function $N$. In Section 3,
we show how these estimates can be used for reconstructing the
absorption coefficient of a small absorbing anomaly.

%
\section{Estimates for Neumann functions}
%

This section is devoted to the study of the Neumann function for
the operator $L$ given by
\[
Lu=\nabla \cdot (\gamma \nabla u) -iku
\]
in a bounded domain $\Omega\subset \RR^d$ with $d\ge 3$. Here, we
assume that $k$ is a positive constant satisfying $k\ge k_0$ for
some $k_0>0$ and $\gamma:\Omega \to \RR$ satisfies the uniform
ellipticity condition
\begin{equation}\label{eqP-02}
\nu \le \gamma(x) \le \nu^{-1},  \quad\forall x\in\Omega,
\end{equation}
for some constant $\nu\in (0,1]$.

We first introduce some (standard) notation and definitions that will be used
throughout the paper. Let $\Omega \in \RR^d$ $(d\geq 3)$ be a bounded Lipschitz domain. We
call $\diam(\Omega)$ the least upper bound of the distances
between pairs of points in $\Omega$. We say that a function $f$ on
$\Omega$ admits a modulus of continuity $\theta$ if
$\theta:\RR_+\to \RR_+$ is a nondecreasing function such that
\[
\abs{f(x)-f(y)} \le \theta(\abs{x-y}),\quad\forall x,y\in\Omega.
\]
For $0< \lambda <1$ and $f\in \mathcal{C}^{0,\lambda}(\Omega)$, we
let $[f]_{0,\lambda;\Omega}$ denote the $\lambda$-H\"older
seminorm of $f$ in $\Omega$; {\it i.e.},
\[
[f]_{0,\lambda;\Omega} = \sup_{x, y\in \Omega; x\neq
y}\frac{\abs{f(x)-f(y)}}{\abs{x-y}^{\lambda}}.
\]

For $ p \geq 1$ and $m$ a non-negative integer, we define the
space $W^{m,p}(\Omega)$ as the family of all $m$ times weakly
differentiable functions in $L^p(\Omega)$, whose weak
derivatives of orders up to $m$ are functions in
$L^p(\Omega)$. We let $W_0^{m,p}(\Omega)$ to be the
closure of $\mathcal{C}^\infty_c(\Omega)$ in
$W^{m,p}(\Omega)$, where $\mathcal{C}^\infty_c(\Omega)$ is the
set of all infinitely differentiable functions with compact
supports in $\Omega$. We use
$\mathcal{C}_{loc}^{0,\lambda}(\Omega)$ and
$W_{loc}^{m,p}(\Omega)$  to denote the local spaces of
functions belonging respectively to
$\mathcal{C}^{0,\lambda}(\Omega^\prime)$ and
$W^{m,p}(\Omega^\prime)$ for all $\Omega^\prime \subset
\subset \Omega$.
We write $u\in L^p(\Omega;\CC)$ (or $u \in W^{m,p}(\Omega;\CC)$, etc.) to emphasize that $u$ is a complex valued function.
We recall that for $m=1$ and $p=2$, the spaces
$W^{1,2}(\Omega;\CC)$ and $W_0^{1,2}(\Omega;\CC)$, equipped with
the inner product
\[
\ip{u, v}:= \int_\Omega \nabla u \cdot \overline{\nabla  v} + u
\overline v \,dx,
\]
are Hilbert spaces.  Finally, for $p>1$ and $q$ being its
conjugate exponent, {\it i.e.}, $1/p + 1/q =1$,  we use
$W^{-1,q}(\Omega;\CC)$ and $W_0^{-1,q}(\Omega;\CC)$ to
respectively denote the dual spaces to $W_0^{1,p}(\Omega;\CC)$ and
$W^{1,p}(\Omega;\CC)$.

Our main result in this section is the following.
\begin{thm}\label{thm27}
Let $\Omega\subset \RR^d$ be a bounded $\mathcal{C}^1$ domain. Let
$\gamma\in \mathcal{C}^{0,\lambda}(\overline\Omega)$ for some
$\lambda \in (0,1)$. Let $N$ be the Neumann function of $L$ in
$\Omega$. For $y\in\Omega$, denote $\gamma_0=\gamma(y)$ and let
$N_0$ be the Neumann function for $L_0=\nabla \cdot \gamma_0
\nabla -ik$ in $\Omega$. Then we have
\begin{equation}\label{eq1.46bc}
\abs{N(x,y)-N_0(x,y)} \le C \abs{x-y}^{2-d+\lambda},\quad \forall
x\in\Omega,\quad x\neq y,
\end{equation}
where $C$ is a constant depending only on $d, \nu, k_0, \lambda,
\Omega$, and  $[\gamma]_{0,\lambda;\Omega}$. Also, if
$\,0<\abs{x-y}<d_y/2$, where $d_y=\dist(y,\partial\Omega)$, then
we  have
\begin{equation}                \label{eq1.59gg}
\abs{\nabla_x ( N(x,y)- N_0(x,y))} \le C
\left(\abs{x-y}^{1-d+\lambda}+ k \abs{x-y}^{3-d+\lambda}\right),
\end{equation}
where the constant $C$ depends on $\diam\Omega$ as well. Moreover,
if we assume further that $\gamma \in
\mathcal{C}^{1,\lambda}(\overline \Omega)$, then for all $x \in
\Omega$ satisfying $0<\abs{x-y}<d_y/2$, we have
\begin{align}               \label{eq2.61hj}
\abs{\nabla_x (N(x,y)- N_0(x,y))} &\le C \abs{x-y}^{1-d+\lambda},\\
                            \label{eq2.60gg}
\abs{\nabla^2_x (N(x,y)- N_0(x,y))} &\le C
\left(\abs{x-y}^{-d+\lambda}+k \abs{x-y}^{2-d+\lambda}\right),
\end{align}
where $C$ depends only on
$\norm{\gamma}_{\mathcal{C}^{1,\lambda}(\Omega)}$, $d, \nu, k_0,
\lambda, \Omega$, and $\diam \Omega$.
\end{thm}

In this section, we first consider the Neumann boundary value
problems for the operators $L$ and its adjoint $\Lt$ given by
 \beq \label{adjointL}
 \Lt:=\nabla \cdot \gamma \nabla + i k.
 \eeq
Then we give a definition of a Neumann function. Next, we
construct Neumann functions, $N$ and $N^*$, of respectively $L$
and $\Lt$ in $\Omega$. Our construction of $N$ and $N^*$ holds for
a Lipschitz bounded domain $\Omega$  and a coefficient $\gamma$
uniformly continuous on $\bar \Omega$. If we further assume that
$\Omega$ is of class $\mathcal{C}^1$, then we are able to derive
$L^p$ estimates for the operators $L$ and $\Lt$ with Neumann
boundary conditions on $\partial \Omega$. Finally, based on the following
global pointwise bound for the Neumann function $N$: \beq
\label{globalb} \abs{N(x,y)} \le C \abs{x-y}^{2-d}\quad \text{for
all }\, x,y\in\Omega\;\text{ with }\;x\ne y,
\end{equation}
where $C$ depends only on $d, \nu, \Omega$, $k_0$, and $\theta$ (a modulus of continuity of $\gamma$), we
describe the local behavior of $N$ such as \eqnref{eq1.46bc}.  Assuming that $\gamma \in
\mathcal{C}^{0,\lambda}(\bar \Omega),$ for $0<\lambda <1$, we
prove that estimates \eqnref{eq1.46bc}--\eqnref{eq2.60gg} hold.

Estimates of \eqnref{globalb}-type were derived for the Dirichlet Green's
function of $L$ with $k=0$ and $\gamma \in L^{\infty}(\Omega)$ in \cite{LSW63, GW82}.
Under the further assumption that the principal coefficients are uniformly continuous of belong to the class VMO, they were generalized to the
vectorial case  in \cite{fuchs,DM, HK07,KK10} and to the periodic case in \cite{AL91,KS11}.

%
\subsection{Neumann boundary value problem}\label{sec:nbp}
%

We begin with the weak formulation of the Neumann boundary value problem
\begin{equation}\label{eq1.00sk}
\left\{
\begin{aligned}
-L u= f +\nabla\cdot F \quad & \text{in }\;\Omega,\\
(\gamma \nabla u + F)\cdot n = g \quad & \text{on }\;\partial\Omega,
\end{aligned}
\right.
\end{equation}
where $f \in L^1_{loc}(\Omega;\CC)$, $F \in
L^1_{loc}(\Omega;\CC^d)$, and $g\in
L^1_{loc}(\partial\Omega;\CC)$. We say that $u\in
W^{1,1}_{loc}(\Omega)$ is a weak solution of problem
\eqref{eq1.00sk} if the following identity holds:
\[
\int_\Omega (\gamma \nabla u \cdot \overline{\nabla  \phi} + ik u
\overline \phi)  \,dx  =\int_\Omega (f \overline \phi - F\cdot
\overline{\nabla \phi}) \,dx +\int_{\partial\Omega} g \overline
\phi \,d\sigma,\quad \forall \phi \in \mathcal{C}^\infty(\overline
\Omega;\CC).
\]
Let $H=W^{1,2}(\Omega;\CC)$. We define the sesquilinear form
$B(\cdot,\cdot):H \times H \to \CC$, associated to the operator
$L$, as
\[
B(u,v):=\int_\Omega (\gamma \nabla u \cdot \overline{\nabla  v} +
ik u \overline v)  \,dx.
\]
It is easy to check that $B$ is bounded and coercive.

Let $f \in L^{2d/(d+2)}(\Omega;\CC)$, $F\in L^2(\Omega;\CC^d)$, and $g \in L^2(\partial\Omega;\CC)$.
Then by the Sobolev embedding and the trace theorem, we find that
\[
\Fcal (v):=\int_\Omega (f \overline v-F\cdot \overline{\nabla
v})\,dx+ \int_{\partial\Omega}  g \overline v\,d\sigma
\]
is a bounded skew-linear functional on $H$.
Therefore, by the Lax-Milgram lemma, we find that there exists a unique $u\in H$ such that
\[
B(u,v)=\Fcal(v),\quad \forall v\in H.
\]
We have thus shown that if $f \in L^{2d/(d+2)}(\Omega;\CC)$, $F\in
L^2(\Omega;\CC^d)$, and $g \in L^2(\partial\Omega;\CC)$, then
problem \eqnref{eq1.00sk} has a unique weak solution $u$ in
$W^{1,2}(\Omega;\CC)$. Since
$\mathcal{C}^\infty(\overline\Omega;\CC)$ is dense in
$W^{1,2}(\Omega;\CC)$, we find that $u$ satisfies following
identity:
\begin{equation}                \label{eq1.05hh}
\int_\Omega (\gamma \nabla u \cdot \overline{\nabla  v} + ik u
\overline v ) \,dx =\int_\Omega (f \overline v- F\cdot
\overline{\nabla v})\,dx +\int_{\partial\Omega} g \overline v
\,d\sigma,\quad \forall v \in W^{1,2}(\Omega;\CC).
\end{equation}

Let $\Lt$ be given by \eqnref{adjointL}. By the same reasoning, we
find that  there exists a unique weak solution $u$ in
$W^{1,2}(\Omega;\CC)$ of problem
\[
\left\{
\begin{aligned}
-\Lt u= f+\nabla \cdot F \quad &\text{in }\;\Omega,\\
(\gamma \nabla u +F)\cdot n=g \quad &\text{on }\;\partial\Omega,
\end{aligned}
\right.
\]
provided $f \in L^{2d/(d+2)}(\Omega;\CC)$ , $F\in L^2(\Omega;\CC^d)$, and $g \in L^2(\partial\Omega;\CC)$; \ie,
\begin{equation}\label{eq1.07tt}
\int_\Omega (\gamma \nabla u \cdot \overline{\nabla  v} - ik u
\overline v ) \,dx  =\int_\Omega (f \overline
v-F\cdot\overline{\nabla v})\,dx+\int_{\partial\Omega} g \overline
v \,d\sigma,\quad \forall v \in W^{1,2}(\Omega;\CC).
\end{equation}

%
\subsection{Definition of the Neumann function} \label{sec:nf}
%

We say that a function $N:\Omega\times\Omega\to
\CC\cup\set{\infty}$ is a Neumann function of $L$ in $\Omega$ if
it satisfies the following properties:
\begin{enumerate}
\item[i)]
$N(\cdot,y) \in W^{1,1}_{loc}(\Omega)$ and $N(\cdot,y) \in W^{1,2}(\Omega\setminus B_r(y))$ for all $y\in\Omega$ and $r>0$.
\item[ii)]
$N(\cdot, y)$ is a weak solution of
\[
\left\{
\begin{aligned}
-L N(\cdot,y) =\delta_y \quad &\text{in }\;\Omega,\\
\gamma \nabla N(\cdot,y)\cdot n = 0 \quad &\text{on }\;\partial\Omega,
\end{aligned}
\right.
\]
for all $y\in\Omega$ in the  sense
\[
\int_{\Omega} (\gamma(x) \nabla_x N(x,y) \cdot \overline{\nabla
\phi(x)} + i k N(x,y) \overline{\phi(x)}) \, dx = \overline
\phi(y),\quad \forall  \phi\in
\mathcal{C}^\infty(\overline\Omega;\CC).
\]
\item[iii)] For any $f \in \mathcal{C}_c^\infty(\Omega; \CC)$, the
function $u$ given by
\begin{equation}        \label{eq2.9x}
u(x):=\int_\Omega \overline N(y,x)  f(y)\,dy
\end{equation}
is the unique solution in $W^{1,2}(\Omega)$ of problem
\begin{equation}                \label{eq2.10yq}
\left\{
\begin{aligned}
-\Lt u&= f \quad\text{in }\;\Omega,\\
\gamma \nabla u \cdot n&=0 \quad \text{on }\;\partial\Omega.
\end{aligned}
\right.
\end{equation}
\end{enumerate}

We remark that part iii) of the above definition gives the
uniqueness of a Neumann function. Indeed, let $\tilde N(x,y)$ be
another function satisfying the above properties. Then by the
uniqueness of a solution in $W^{1,2}(\Omega;\CC)$ of problem
\eqref{eq2.10yq}, we have
\[
\int_\Omega (\overline{N} - \overline{\tilde N})(y,x)
f(y)\,dy=0,\quad \forall  f \in \mathcal{C}_c^\infty(\Omega; \CC),
\]
and thus we conclude that $N = \tilde N$ {\it a.e.} in
$\Omega\times\Omega$.

\subsection{Local boundedness estimates}
Let $B_R=B_R(x_0)$ be the ball of radius $R$ centered at $x_0$, and let $u\in W^{1,2}(B_R)$ be a weak solution of $-Lu=0$ in $B_R$.
For $0<\rho<R$, let $\eta$ be a smooth cut-off function satisfying
\[
0\le \eta \le 1,\quad \supp \eta \subset B_R,\quad \eta\equiv
1\;\text{ on }\; B_\rho, \quad\text{and}\quad \abs{\nabla \eta}
\le 2/(R-\rho).
\]
By taking $\eta^2 \overline u$ as a test function, we get
\[
\int_{B_R} \gamma \eta^2 \abs{\nabla u}^2\,dx=-\int_{B_R}
2\gamma\eta \overline u \nabla u \cdot \nabla \eta \,dx
-ik\int_{B_R} \eta^2\abs{u}^2\,dx.
\]
By taking real parts in the above and using Cauchy's inequality, we get
\begin{equation}\label{eq1.09ww}
\int_{B_R} \gamma \eta^2 \abs{\nabla u}^2 = -\Re \int_{B_R} 2\gamma \eta \overline u \nabla u \cdot \nabla \eta \,dx \le \frac{1}{2} \int_{B_R} \gamma \eta^2\abs{\nabla u}^2\,dx+ 2\int_{B_R} \gamma \abs{\nabla \eta}^2 \abs{u}^2\,dx.
\end{equation}
Therefore, we obtain Caccioppoli's inequality
\begin{equation}\label{eq1.10ax}
\int_{B_\rho} \abs{\nabla u}^2\,dx \le \frac{C}{(R-\rho)^2} \int_{B_R} \abs{u}^2\,dx,
\end{equation}
where $C=C(\nu)$.

Next, we consider the operator $L_0$ defined by
\begin{equation}                    \label{eq1.14tm}
L_0 u= \nabla \cdot(\gamma_0  \nabla u) - iku=\gamma_0\Delta
u-iku,
\end{equation}
where $\gamma_0$ is a constant satisfying the condition \eqref{eqP-02}.
Let $u\in W^{1,2}(B_1)$ be a weak solution of $-L_0 u =0$.
Since $L_0$ has constant coefficients, we may apply \eqref{eq1.10ax} to derivatives of $u$ iteratively to get
\[
\norm{u}_{W^{m,2}(B_{1/2})} \le C(m,\nu) \norm{u}_{L^2(B_1)},\quad m=1,2,\ldots.
\]
By the Sobolev embedding theorem, we then have
\[
\sup_{B_{1/2}}\, \abs{u} \le C(d) \norm{u}_{W^{m,2}(B_{1/2})} \le C(d,\nu) \norm{u}_{L^2(B_1)},
\]
where $m=[d/2]+1$. Here and throughout this paper $[s]$ denotes the smallest integer not less than $s$. Since the above estimate does not depend on $k$, by a scaling argument we conclude that if $u\in W^{1,2}(B_R)$ is a weak solution of $-L_0 u=0$ in $B_R$, then we have
\[
\sup_{B_{R/2}}\, \abs{u} \le C(d,\nu) R^{-d/2} \norm{u}_{L^2(B_R)}.
\]
Similarly, if $u\in W^{1,2}(B_R)$ is a weak solution of $-L_0 u=0$ in $B_R$, then we have
\begin{equation}\label{eq3rd}
\sup_{B_{R/2}}\, \abs{\nabla u} \le C R^{-d/2} \norm{\nabla u}_{L^2(B_R)}.
\end{equation}
It follows from the above estimate that for all $0<\rho<r \le R$, we have
\begin{equation}\label{eq1.11yn}
\int_{B_\rho} \abs{\nabla u}^2\,dx \le C(\rho/r)^d \int_{B_r} \abs{\nabla u}^2 \,dx,
\end{equation}
where $C=C(d,\nu)$.
Indeed, in the case when $\rho< r/2$, we utilize \eqref{eq3rd} to get the above estimate; otherwise, then we may simply take $C=2^d$ in \eqref{eq1.11yn}.

Observe that the same estimates are valid for $u\in W^{1,2}(B_R)$ satisfying
 $-\Lt_0u=0$ weakly in $B_R$, where $\Lt_0$ is defined as $\Lt_0 = \nabla \cdot \gamma_0 \nabla + ik$.

\begin{lem}                 \label{lem1.1ap}
Assume that $\gamma \in \mathcal{C}^0(\overline\Omega)$ and let
$\theta$ be a modulus of continuity of $\gamma$. Let $B_R=B_R(x_0)
\subset \Omega$ and let $u\in W^{1,2}(B_R;\CC)$ be a weak solution
of either $-L u =f+\nabla \cdot F$ or $-\Lt u=f+\nabla \cdot F$ in
$B_R$, where $f\in L^q(B_R;\CC)$ with $q>d/2$ and $F\in
L^p(B_R;\CC^d)$ with $p>d$. Then $u$ is locally H\"older
continuous in $B_R$ and the following estimate holds:
\begin{equation}                    \label{eq1.13vv}
R^{\lambda_0}[u]_{0,\lambda_0;B_{R/2}} \le C_0
\left(R^{-d/2}\norm{u}_{L^2(B_R)}+ R^{2-d/q}\norm{f}_{L^q(B_R)} +
R^{1-d/p}\norm{F}_{L^p(B_R)}\right),
\end{equation}
where $\lambda_0 \in (0,1)$ and $C_0$ are constants depending on
$d, \nu, p, q$, and $\theta$, and $[u]_{0,\lambda_0;D}$ denotes
the $\lambda_0$-H\"older seminorm of $u$ in $D$. Moreover, for any
$p_0>0$ and $0<\rho<R$, we have
\begin{equation}                    \label{eq1.14vt}
\sup_{B_\rho} \,\abs{u} \le C \left((R-\rho)^{-d/p_0}\left(\int_{B_R}\abs{u}^{p_0}\,dx\right)^{1/p_0}
+R^{2-d/q}\norm{f}_{L^q(B_R)}+R^{1-d/p}\norm{F}_{L^p(B_R)}\right),
\end{equation}
where $C$ depends on $d, \nu, p, q, p_0$, and $\theta$.
\end{lem}
\begin{proof}
We consider the case when $u$ is a weak solution of
\begin{equation}\label{eq1.19uu}
-Lu=f+\nabla \cdot F\quad\text{in }\;B_R.
\end{equation}
The proof for the other case is identical.
Let $R_0>0$ be a number to be
fixed later. Let $y\in B_R$ and $0<r\le R_0$ be arbitrary but
fixed. Denote $\gamma_0=\gamma(y)$ and let $L_0$ be defined as in
\eqref{eq1.14tm}. Observe that $u$ is a weak solution of
\[
-L_0 u=f+\nabla \cdot F +\nabla\cdot ((\gamma-\gamma_0)\nabla
u)\quad\text{in }\;B_R.
\]
Let $w\in W^{1,2}_0(B_r(y))$ be the unique weak solution of
\[
\left\{
\begin{aligned}
-L_0 w&=f+\nabla\cdot F +\nabla\cdot ((\gamma-\gamma_0)\nabla u)\quad\text{in }\;B_r(y),\\
w &=0 \quad \text{on }\;\partial B_r(y).
\end{aligned}
\right.
\]
Then $w$ satisfies the following identity:
\[
\int_{B_r(y)} (\gamma_0 |\nabla w|^2 +i k |w|^2
)\,dx=\int_{B_r(y)} (f \overline w- F\cdot \overline{\nabla w} -
(\gamma-\gamma_0) \nabla u\cdot \overline{\nabla w}) \,dx.
\]
Taking the real parts in the above and using Sobolev embedding,
Poincar\'e inequality, and H\"older's inequalities, we may deduce
that
\begin{multline*}
\int_{B_r(y)} \gamma_0 \abs{\nabla w}^2\,dx \le C r^{d/2+1-d/q}\norm{f}_{L^q(B_r(y))}\norm{\nabla w}_{L^2(B_r(y))}\\
+C r^{d/2-d/p} \norm{F}_{L^p(B_r(y))}\norm{\nabla w}_{L^2(B_r(y))}
+\theta(r) \norm{\nabla u}_{L^2(B_r(y))} \norm{\nabla
w}_{L^2(B_r(y))}.
\end{multline*}
Denote $\lambda_1=2-d/q$ and $\lambda_2=1-d/p$. From the above
inequality, we obtain
\[
\norm{\nabla w}_{L^2(B_r(y))}^2 \le
Cr^{d-2+2\lambda_1}\norm{f}_{L^q(B_r(y))}^2+
Cr^{d-2+2\lambda_2}\norm{F}_{L^p(B_r(y))}^2+ C \theta(r)^2
\norm{\nabla u}_{L^2(B_r(y))}^2 .
\]
On the other hand, observe that $v:=u-w$ satisfies $-L_0 v=0$ weakly in $B_r(y)$.
Therefore, by \eqref{eq1.11yn}, for $0<\rho<r$, we get
\begin{align*}
& \int_{B_\rho(y)} \abs{\nabla u}^2\,dx \le 2 \int_{B_\rho(y)} \abs{\nabla v}^2\,dx+2\int_{B_\rho(y)} \abs{\nabla w}^2\,dx \\
& \le C\left(\frac{\rho}{r}\right)^d  \int_{B_r(y)} \abs{\nabla v}^2\,dx + 2\int_{B_r(y)} \abs{\nabla w}^2\,dx\\
& \le C\left(\frac{\rho}{r}\right)^d  \int_{B_r(y)} \abs{\nabla u}^2\,dx+C\int_{B_r(y)} \abs{\nabla w}^2\,dx\\
&\le  C\left[\left(\frac{\rho}{r}\right)^d+\theta(r)^2\right]
\int_{B_r(y)} \abs{\nabla u}^2\,dx+
Cr^{d-2+2\lambda_1}\norm{f}_{L^q(B_r(y))}^2+
Cr^{d-2+2\lambda_2}\norm{F}_{L^p(B_r(y))}^2.
\end{align*}
By Campanato's iteration argument (see, for instance,
\cite[Lemma~2.1, p. 86]{Gi83}), we find that if $\theta(R_0)$ is
small enough, then for all  $0<\rho<r\le R_0$ we have
\begin{multline*}
\int_{B_\rho(y)} \abs{\nabla u}^2 \,dx\le C\left(\frac{\rho}{r}\right)^{d-2+2\lambda_0} \int_{B_r(y)} \abs{\nabla u}^2\,dx\\
+C\rho^{d-2+2\lambda_0}r^{2(\lambda_1-\lambda_0)}\norm{f}_{L^q(B_r(y))}^2+
C\rho^{d-2+2\lambda_0}r^{2(\lambda_2-\lambda_0)}\norm{F}_{L^p(B_r(y))}^2,
\end{multline*}
where $0<\lambda_0<\min(\lambda_1,\lambda_2)=\min(2-d/q,1-d/p)$.
The above estimate (via Morrey's characterization of H\"older
continuous functions in terms of Dirichlet integrals; see, for
instance, \cite[Theorem~3.5.2]{Morrey}) implies that $u$ is
locally H\"older continuous in $B_R$ and, in particular, we have
the estimate
\begin{equation}\label{eq:bene}
R^{2\lambda_0}[u]_{0,\lambda_0;B_{R/4}}^2 \le  C
\left(R^{2-d}\int_{B_{R/2}}\abs{\nabla u}^2\,dx+
R^{2(2-d/q)}\norm{f}_{L^q(B_R)}^2+R^{2(1-d/p)}\norm{F}_{L^p(B_R)}^2\right).
\end{equation}

Let $\eta$ be a smooth cut-off function satisfying
\[
0\le \eta \le 1,\quad \supp \eta \subset B_R,\quad \eta\equiv
1\;\text{ on }\; B_{R/2}, \quad\text{and}\quad \abs{\nabla \eta}
\le 4/R.
\]
By taking $\eta^2 \overline u$ as a test function in \eqref{eq1.19uu}, we get
\begin{multline*}
\int_{B_R} \gamma \eta^2 \abs{\nabla u}^2\,dx+ik\int_{B_R} \eta^2\abs{u}^2\,dx\\
=-\int_{B_R} 2\gamma\eta \overline u \nabla u \cdot \nabla \eta \,dx +\int_{B_R} \eta^2
f \overline u\,dx+ \int_{B_R} \eta^2  F\cdot \overline{\nabla u}\,dx
+ \int_{B_R} 2 \eta \overline u F\cdot \nabla \eta\,dx.
\end{multline*}
By taking the real parts in the above and using Cauchy's
inequality, we get
\[
\int_{B_{R/2}} \abs{\nabla u}^2\,dx \le C R^{-2} \int_{B_R}  \abs{u}^2 \,dx+CR^2 \int_{B_R} \abs{f}^2\,dx
+C \int_{B_R} \abs{F}^2\,dx.
\]
By H\"older's inequality, we then obtain
\[
\int_{B_{R/2}} \abs{\nabla u}^2\,dx \le C R^{-2} \int_{B_R}  \abs{u}^2 \,dx+CR^{2+d-2d/q} \norm{f}_{L^q(B_R)}^2
+CR^{d-2d/p} \norm{F}_{L^p(B_R)}^2.
\]
By combining \eqref{eq:bene} and the above inequality, we get \eqref{eq1.13vv} via a standard covering argument.

Observe that for any $x \in B_{R/2}$, we have
\[
\abs{u(x)} \le \abs{u(x')}+ \abs{u(x)-u(x')} \le \abs{u(x')} +
R^{\lambda_0}[u]_{0,\lambda_0;B_{R/2}},\quad \forall x'\in
B_{R/2}.
\]
By taking average with respect to $x'$ in $B_{R/2}$ and then using \eqref{eq1.13vv} and H\"older's inequality we get
\[
\sup_{B_{R/2}}\,\abs{u} \le C\left(R^{-d/2}\norm{u}_{L^2(B_R)}+R^{2-d/q}\norm{f}_{L^q(B_R)}+R^{1-d/p}\norm{F}_{L^p(B_R)} \right).
\]
By using a standard iteration argument (see \cite[pp. 80--82]{Gi93}), we obtain  \eqref{eq1.14vt} from the above inequality. This completes the proof. \qed
\end{proof}

\subsection{Construction of Neumann functions}
The aim of this subsection is to construct Neumann functions of
$L$ and $\Lt$ in $\Omega$ and derive their basic properties. The
following theorem holds.
\begin{thm}\label{thm1ap}
Assume $\gamma \in \mathcal{C}^0(\overline\Omega)$. Then there
exist Neumann functions $N(x,y)$ and $N^* (x,y)$ of respectively
$L$ and $\Lt$ in $\Omega$. Moreover, there exists $\lambda_0\in
(0,1)$ such that $N(\cdot,y),\, N^*(\cdot,y) \in
\mathcal{C}^{0,\lambda_0}_{loc}(\Omega\setminus \set{y})$ for all
$y\in\Omega$ and the identity,
\begin{equation}                    \label{eq3.01mq}
N^* (x,y):= \overline N(y,x),\quad \forall x, y \in \Omega,\;\;
x\ne y,
\end{equation}
holds. Furthermore, the following estimates hold uniformly in
$y\in\Omega$, where we denote $d_y=\dist(y,\partial\Omega)$:
\begin{enumerate}
\item[{\rm i)}] $\norm{N(\cdot,y)}_{L^{2d/(d-2)}(\Omega\setminus
B_r(y))}+\norm{\nabla N(\cdot,y)}_{L^2(\Omega\setminus B_r(y))}
\le C r^{1-d/2}\;$ for all $r \in (0,d_y)$. \item[{\rm ii)}]
$\norm{N(\cdot,y)}_{L^p(B_r(y))}\le C r^{2-d+d/p}\;$ for all $r\in
(0,d_y)$, where $p\in [1,\frac{d}{d-2})$. \item[{\rm iii)}]
$\abs{\set{x\in\Omega:\abs{N(x,y)}>t}}\le C t^{-d/(d-2)}\;$ for
all $t> d_y^{2-d}$. \item[{\rm iv)}] $\norm{\nabla
N(\cdot,y)}_{L^p(B_r(y))}\le C r^{1-d+d/p}\;$ for all $r\in (0,
d_y)$, where $p\in [1,\frac{d}{d-1})$. \item[{\rm v)}]
$\abs{\set{x\in\Omega:\abs{\nabla_x N(x,y)}>t}}\le C
t^{-d/(d-1)}\;$ for all $t> d_y^{1-d}$. \item[{\rm vi)}]
$\abs{N(x,y)}\le C \abs{x-y}^{2-d}\;$ whenever
$0<\abs{x-y}<d_y/2$. \item[{\rm vii)}] $\abs{N(x,y)-N(x',y)} \le C
\abs{x-x'}^{\lambda_0} \abs{x-y}^{2-d-\lambda_0}\;$ if
$2\abs{x-x'}<\abs{x-y}<d_y/2$.
\end{enumerate}
In the above, $C$ is a constant depending on $d, \nu, k_0, \Omega$,
and $\theta$; it depends on $p$ as well in ii) and  iv). The
estimates i) -- vii) are also valid for $N^*(x,y)$. Finally, if
$q>d/2$ and $p>d$, then for any $f \in L^q(\Omega, \CC)$, $F \in
L^p(\Omega;\CC^d)$ and $g\in L^2(\partial\Omega;\CC)$, the
function $u$ given by
\begin{equation}        \label{eqM1.e}
u(x):=\int_\Omega \left( N(x,y) f(y) -\nabla_y N(x,y) \cdot F(y) \right) \,dy+ \int_{\partial\Omega} N(x,y) g(y)\,d\sigma(y)
\end{equation}
is the unique solution in $W^{1,2}(\Omega)$ of problem
\eqref{eq1.00sk}.

\end{thm}

\begin{proof}
We follow the strategy used in \cite{CK11}, which in turn is based
on \cite{HK07}. Let us fix a function $\Phi \in
\mathcal{C}_c^\infty(\RR^d)$ such that $\Phi$ is supported in
$B_1(0)$, $0\le \Phi \le 2$, and $\int_{\RR^d} \Phi \;dx=1$. Let
$y\in \Omega$ be fixed but arbitrary. For any $\epsilon>0$, we
define
\[
\Phi_\epsilon(x)=\epsilon^{-d}\Phi((x-y)/\epsilon).
\]
Let $v_{\epsilon, y}$ be the unique weak solution in
$W^{1,2}(\Omega;\CC)$ of problem
\begin{equation}                    \label{eq4.0ik}
\left\{
\begin{aligned}
-L v = \Phi_\epsilon \quad &\text{in }\;\Omega,\\
\gamma \nabla v \cdot  n =0 \quad &\text{on }\; \partial\Omega.
\end{aligned}
\right.
\end{equation}
We define the ``averaged Neumann function'' $N^\epsilon(\cdot,y)$ by
\[
N^\epsilon(\cdot,y)=v=v_{\epsilon,y}.
\]
Then $N^\epsilon(\cdot,y)$ satisfies the following identity (c.f. \eqref{eq1.05hh}):
\begin{equation}                \label{eq4.1tv}
\int_\Omega (\gamma \nabla N^\epsilon(\cdot,y) \cdot \overline{
\nabla \phi} + ik N^\epsilon(\cdot,y)\overline \phi)\,dx =
\int_{\Omega\cap B_\epsilon(y)} \Phi_\epsilon \overline \phi
\,dx,\quad \forall \phi \in W^{1,2}(\Omega;\CC).
\end{equation}
By taking $\phi=N^\epsilon(\cdot,y)=v$ in \eqref{eq4.1tv}, we get
\[
\int_\Omega \gamma \abs{\nabla v}^2 dx = \Re \int_\Omega (\gamma
\abs{\nabla v}^2 + i k \abs{v}^2) \,dx  = \Re \int_{\Omega \cap
B_\epsilon(y)} \Phi_\epsilon \overline v\,dx \le C
\epsilon^{(2-d)/2}\norm{v}_{W^{1,2}(\Omega)},
\]
where the last inequality follows from the Sobolev embedding, namely,
 $$
 \left| \int_{\Omega \cap B_\epsilon(y)} \Phi_\epsilon \overline v\,dx \right| \le C \norm{\Phi_\epsilon}_{L^{2d/(d+2)}(B_\epsilon(y))}\norm{v}_{W^{1,2}(\Omega)} \le C \epsilon^{(2-d)/2}\norm{v}_{W^{1,2}(\Omega)} .
 $$
Similarly, we get
\[
\int_\Omega k \abs{v}^2 dx = \Im \int_\Omega (\gamma \abs{\nabla
v}^2 + i k \abs{v}^2) \,dx  = \Im \int_ {\Omega \cap
B_\epsilon(y)} \Phi_\epsilon \overline v\,dx \le C
\epsilon^{(2-d)/2}\norm{v}_{W^{1,2}(\Omega)}.
\]
Therefore, we have
\begin{equation}        \label{eqG-02}
\norm{N^\epsilon(\cdot,y)}_{W^{1,2}(\Omega)}\le C \epsilon^{(2-d)/2},
\end{equation}
where $C=C(d,\nu, k_0)$.

Let $R\in (0,d_y)$ be arbitrary, but fixed. Assume that $f \in
\mathcal{C}_c^\infty(\Omega;\CC)$ is supported in $B_R=B_R(y)
\subset \Omega$. Let $u$ be a unique weak solution in
$W^{1,2}(\Omega;\CC)$ of problem \eqref{eq2.10yq}. We then have
the following identity (c.f. \eqref{eq1.07tt}):
\begin{equation}                \label{eq4.2ym}
\int_\Omega (\gamma \nabla w \cdot \overline{\nabla u}+ ik w
\overline u)\,dx =\int_\Omega w \overline f\,dx,\quad \forall w
\in W^{1,2}(\Omega;\CC).
\end{equation}
Then by setting $\phi=u$ in \eqref{eq4.1tv} and setting $w=N^\epsilon(\cdot,y)=v$ in \eqref{eq4.2ym}, we get
\begin{equation}\label{eq4.3wx}
\int_\Omega N^\epsilon(x,y) \overline{f(x)} \,dx= \int_ {\Omega\cap B_\epsilon(y)} \Phi_\epsilon \overline u \, dx.
\end{equation}
Also, by taking $w=u$ in \eqref{eq4.2ym}, we see that
\[
\int_\Omega \gamma \abs{\nabla u}^2\,dx +ik\int_\Omega
\abs{u}^2\,dx=\int_\Omega u\overline f\,dx.
\]
Taking the real and imaginary parts in the above and using the
Sobolev embedding  and H\"older's inequality
\begin{align*}
\int_\Omega \gamma \abs{\nabla u}^2\,dx &= \Re \int_\Omega u \overline f\,dx \le C \norm{f}_{L^{2d/(d+2)}(\Omega)}\norm{u}_{W^{1,2}(\Omega)},\\
k \int_\Omega \abs{u}^2\,dx &= \Im \int_\Omega u \overline f\,dx \le C \norm{f}_{L^{2d/(d+2)}(\Omega)}\norm{u}_{W^{1,2}(\Omega)}.
\end{align*}
Therefore, we obtain
\begin{equation}                    \label{eq4.4ur}
\norm{u}_{W^{1,2}(\Omega)}\le C \norm{f}_{L^{2d/(d+2)}(\Omega)},
\end{equation}
where $C=C(d,\nu, k_0)$. From \eqref{eq1.14vt} in Lemma~\ref{lem1.1ap} with $p_0=2d/(d-2)$,
it follows that
\[
\norm{u}_{L^\infty(B_{R/2})} \le C \left(R^{1-d/2} \norm{u}_{L^{2d/(d-2)}(\Omega)}+ R^2 \norm{f}_{L^\infty(B_R)}\right).
\]
Furthermore,  \eqref{eq4.4ur} yields
\[
\norm{u}_{L^{2d/(d-2)}(\Omega)} \le C
R^{1+d/2}\norm{f}_{L^\infty(B_R)},
\]
provided that $f$ is supported in $B_R$. Therefore, by combining
the above two inequalities, we have
\begin{equation}                    \label{eq:starbucks}
\norm{u}_{L^\infty(B_{R/2})} \le C R^2 \norm{f}_{L^\infty(B_R)},
\end{equation}
where $C$ depends on $d, \nu, \Omega$, and $\theta$. By
\eqref{eq4.3wx} and \eqref{eq:starbucks}, we find that for all
$\epsilon \in (0,R/2)$ and $R\in(0,d_y)$,
\[
\Biggabs{\int_{B_R} N^\epsilon(\cdot,y)\overline f\,dx} \le C
R^{2} \norm{f}_{L^\infty(B_R)},\;\; \forall f\in
\mathcal{C}_c^\infty(B_R;\CC).
\]
Therefore, by duality, we conclude that
\[
\norm{N^\epsilon (\cdot,y)}_{L^1(B_R(y))}\le C R^2,\quad \forall \epsilon \in(0,R/2),\;\;  \forall R\in(0,d_y).
\]

Now, for any $x\in\Omega$ such that $0<\abs{x-y}<d_y/2$, let us take $R:=2\abs{x-y}/3$.
Notice that if $\epsilon<R/2$, then $N^\epsilon(\cdot,y)\in W^{1,2}(B_R(x))$ and satisfies $-L N^\epsilon(\cdot,y)= 0$ in $B_R(x)$.
Then by \eqref{eq1.14vt} in Lemma~\ref{lem1.1ap}, we have
\[
\abs{N^\epsilon(x,y)} \le C r^{-d} \norm{N^\epsilon(\cdot,y)}_{L^1(B_r(x))} \le C r^{-d} \norm{N^\epsilon(\cdot,y)}_{L^1(B_{3r}(y))}\le C r^{2-d}.
\]
We have thus shown that for any $x, y\in\Omega$ satisfying $0<\abs{x-y}<d_y/2$, we have
\begin{equation}        \label{eq4.5ac}
\abs{N^\epsilon(x,y)} \le C \abs{x-y}^{2-d},\quad \forall \epsilon <\abs{x-y}/3.
\end{equation}

Next, fix $r\in (0,d_y/2)$ and $\epsilon \in(0,r/6)$.
Let $\eta$ be a smooth function on $\RR^d$ satisfying
\begin{equation}            \label{eq4.19h}
0\le \eta\le 1,\quad \eta\equiv 1\,\text{ on }\,\RR^d\setminus B_r(y),\quad \eta\equiv 0\,\text{ on }\, B_{r/2}(y),\quad\text{and}\quad \abs{\nabla\eta} \le 4/r.
\end{equation}
We set $\phi=\eta^2 v=\eta^2 N^\epsilon(\cdot, y)$ in \eqref{eq4.1tv} to get
\[
\int_\Omega \gamma \eta^2 \abs{\nabla v}^2 \,dx +ik \int_\Omega \eta^2 \abs{v}^2\,dx= -\int_\Omega
2\gamma \eta \overline v \nabla v \cdot \nabla \eta  \,dx.
\]
By taking the real part in the above and using Cauchy's inequality, we get (c.f. \eqref{eq1.09ww})
\[
\int_\Omega \gamma \eta^2 \abs{\nabla N^\epsilon(x, y)}^2 \,dx \le  4\int_\Omega \gamma \abs{\nabla \eta}^2
\abs{N^\epsilon(x, y)}^2\,dx.
\]
We then use \eqref{eq4.5ac} to obtain
\[
\int_\Omega \eta^2 \abs{\nabla N^\epsilon(x, y)}^2\,dx \le C r^{-2} \int_{B_r(y)\setminus B_{r/2}(y)} \abs{x-y}^{2(2-d)}\,dx\le  C r^{2-d}.
\]
Therefore, for all $0<\epsilon<r/6$, we have
\[
\norm{\nabla N^\epsilon(\cdot,y)}_{L^2(\Omega\setminus B_r(y))}\le C r^{(2-d)/2}.
\]
In the case when $\epsilon \ge r/6$, we obtain from \eqref{eqG-02} that
\[
\norm{\nabla N^\epsilon(\cdot,y)}_{L^2(\Omega\setminus B_r(y))} \le \norm{\nabla N^\epsilon(\cdot,y)}_{L^2(\Omega)} \le C r^{(2-d)/2}.
\]
By combining the above two inequalities, we obtain
\begin{equation}                            \label{eqG-14}
\norm{\nabla N^\epsilon(\cdot,y)}_{L^2(\Omega\setminus B_r(y))}\le C r^{(2-d)/2},\quad \forall r\in(0,d_y/2),\quad \forall \epsilon>0.
\end{equation}
Observe that \eqref{eq4.5ac} also implies
\[
\norm{N^\epsilon(\cdot,y)}_{L^{2d/(d-2)}(\Omega\setminus B_r(y))} \le C r^{(2-d)/2},\quad \forall \epsilon \in (0,r/6).
\]
On the other hand, if $\epsilon\ge r/6$, then \eqref{eqG-02} implies
\[
\norm{N^\epsilon(\cdot,y)}_{L^{2d/(d-2)}(\Omega\setminus B_{r}(y))} \le C \norm{N^\epsilon(\cdot,y)}_{W^{1,2}(\Omega)} \le C r^{(2-d)/2}.
\]
By combining the above two estimates, we obtain
\begin{equation}                            \label{eqG-20}
\norm{ N^\epsilon(\cdot,y)}_{L^{2d/(d-2)}(\Omega\setminus B_r(y))} \le C r^{(2-d)/2}, \quad \forall r \in (0,d_y/2),\quad \forall \epsilon>0.
\end{equation}
From the obvious fact that $d_y/2$ and $d_y$ are comparable to
each other, we find by \eqref{eqG-14} and \eqref{eqG-20} that for
all $0<r<d_y$ and $\epsilon>0$, we have
\begin{equation}                    \label{eq4.11bs}
\norm{N^\epsilon(\cdot,y)}_{L^{2d/(d-2)}(\Omega\setminus B_r(y))}+\norm{\nabla N^\epsilon(\cdot,y)}_{L^2(\Omega\setminus B_r(y))}\le C r^{(2-d)/2}.
\end{equation}
From \eqref{eq4.11bs} it follows that (see \cite[pp.
147--148]{HK07})
\begin{align}
\label{eq4.12eq}
\abs{\set{x\in\Omega:\abs{N^\epsilon(x,y)}>t}}\le C t^{-d/(d-2)},\quad  \forall t> d_y^{2-d},\;\; \forall \epsilon>0,\\
\label{eq4.13tz}
\abs{\set{x\in\Omega:\abs{\nabla_x N^\epsilon(x,y)}>t}} \le C t^{-d/(d-1)},\quad \forall t> d_y^{1-d},\;\; \forall \epsilon>0.
\end{align}
It is routine to derive the following strong type estimates from
the above weak type estimates \eqref{eq4.12eq} and
\eqref{eq4.13tz} (see, for instance, \cite[p.~148]{HK07}):
\begin{align}
\label{eq4.14mq}
\norm{N^\epsilon(\cdot,y)}_{L^p(B_r(y))}\le C r^{2-d+d/p},\;\; \forall r\in (0,d_y), \; \forall \epsilon>0,\;\forall p \in[1,\tfrac{d}{d-2}),\\
\label{eq4.15tw}
\norm{\nabla N^\epsilon(\cdot,y)}_{L^p(B_r(y))}\le C r^{1-d+d/p},\;\; \forall r\in (0,d_y),\; \forall \epsilon>0,\;\forall p \in[1,\tfrac{d}{d-1}).
\end{align}

From \eqref{eqG-14}, \eqref{eq4.14mq}, and \eqref{eq4.15tw}, it
follows that  there exists a sequence $\{\epsilon_n
\}_{n=1}^\infty$ tending to zero and a function $N(\cdot,y)$ such
that $N^{\epsilon_n}(\cdot,y)$ converges to $N(\cdot,y)$ weakly
in $W^{1,p}(B_r(y))$ for $1<p<d/(d-1)$ and all $r\in(0,d_y)$ and
also that $N^{\epsilon_n}(\cdot,y)$ converges to $N(\cdot,y)$
weakly in $W^{1,2}(\Omega\setminus B_r(y))$ for all $r\in(0,d_y)$;
see \cite[p.~159]{HK07} for the details. Then it is routine to
check that $N(\cdot,y)$ satisfies the properties i) and ii) at the
beginning of Section~\ref{sec:nf}, and also the estimates i) -- v)
in the theorem; see \cite[Section~4.1]{HK07}.

We now turn to the pointwise bound for $N(x,y)$.
For any $x\in\Omega$ such that $0<\abs{x-y}<d_y/2$, set $R:=2\abs{x-y}/3$.
Notice that \eqref{eq4.11bs} implies that $N(\cdot,y)\in W^{1,2}(B_R(x))$ and satisfies $-L N(\cdot,y)= 0$ weakly in $B_R(x)$.
Then, by \eqref{eq1.14vt} in Lemma~\ref{lem1.1ap} and the estimate ii) in the theorem, we have
\[
\abs{N(x,y)} \le C R^{-d} \norm{N(\cdot,y)}_{L^1(B_R(x))} \le C R^{-d} \norm{N(\cdot,y)}_{L^1(B_{3R}(y))} \le C\abs{x-y}^{2-d}.
\]
We have thus shown that the estimate vi) in the theorem holds.
Then, it is routine to see that the estimate vii) in the theorem follows from \eqref{eq1.13vv} in Lemma~\ref{lem1.1ap} and the above estimate.

Next, let $x\in \Omega\setminus\set{y}$ be fixed but arbitrary,
and let $\tilde N^{\epsilon'}(\cdot,x)\in W^{1,2}(\Omega;\CC)$ be the
averaged Neumann function of the adjoint operator $\Lt$ in $\Omega$, where $0<\epsilon'<d_x$.
Then we have
\begin{equation}\label{eq1.36tv}
\int_{\Omega} (\gamma \nabla \tilde N^{\epsilon'}(z,x) \cdot
\overline{ \nabla \psi}(z) - ik \tilde N^{\epsilon'}(z,x)\overline
\psi(z)) \,dz = \int_ {\Omega\cap B_{\epsilon'}(x)}
\Phi_{\epsilon'}(z) \overline \psi (z) \,dz,
\end{equation}
for all $\psi \in W^{1,2}(\Omega;\CC)$. By setting $\phi=\tilde N^{\epsilon'}(\cdot,x)$ in \eqref{eq4.1tv} and $\psi=N^\epsilon(\cdot,y)$ in \eqref{eq1.36tv} and then taking complex conjugate, we obtain
\[
\int_ {\Omega\cap B_{\epsilon'}(x)} \Phi_{\epsilon'} N^\epsilon(\cdot,y)\,dz= \int_ {\Omega\cap B_\epsilon(y)} \Phi_\epsilon \overline{\tilde N^{\epsilon'}(\cdot,x)}\,dz.
\]
Let $N^*(\cdot,x)$ be a Neumann function of $\Lt$ in $\Omega$ obtained from $\tilde N^{\epsilon_m}(\cdot,x)$, where  $\set{\epsilon_m}_{m=1}^\infty$ is a sequence tending to $0$.
Then, by following the same steps as in \cite[p.~151]{HK07}, we conclude
\[
N(x,y)=\overline {N^*(y,x)},
\]
which obviously implies the identity \eqref{eq3.01mq}.
We remark that by following similar lines of reasoning as in \cite[p.~151]{HK07}, we find
\[
N^\epsilon(x,y)=\epsilon^{-d}\int_{\Omega}\Phi\left(\frac{z-y}{\epsilon}\right) N(x,z)\,dz,
\]
and thus we have in fact the following pointwise convergence:
\begin{equation}
\label{eq4.17ht}
\lim_{\epsilon \to 0} N^\epsilon(x,y)= N(x,y), \quad \forall x,y\in\Omega,\;\; x\ne y.
\end{equation}

Now, let $u$ be the unique solution in $W^{1,2}(\Omega;\CC)$ of
problem \eqref{eq2.10yq} with $f\in
\mathcal{C}_c^\infty(\Omega;\CC)$. By Lemma~\ref{lem1.1ap}, we
find that $u$ is continuous in $\Omega$. By setting
$w=N^\epsilon(\cdot,y)$ in \eqref{eq4.2ym} and setting $\phi= u$
in \eqref{eq4.1tv}, we get
\[
\int_{\Omega} N^\epsilon(x,y) \overline f(x) \,dx = \int_ {\Omega\cap B_\epsilon(y)} \Phi_\epsilon \overline u\,dx.
\]
We take the limit $\epsilon\to 0$ above and then take complex conjugate to get
\[
u(y)=\int_\Omega \overline N(x,y) f(x)\,dx,
\]
which is equivalent to \eqref{eq2.9x}. We have shown that $N(x,y)$
satisfies the property iii) in Section~\ref{sec:nf}, and thus that
$N(x,y)$ is the unique Neumann function of the operator $L$ in
$\Omega$.

Finally, let $f\in L^q(\Omega;\CC)$ with $q>d/2$ and $g\in
L^2(\partial\Omega;\CC)$, and let $u$ be the unique weak solution
in $W^{1,2}(\Omega;\CC)$ of problem \eqref{eq1.00sk}; see
Section~\ref{sec:nbp}. Then $u$ satisfies the identity
\eqref{eq1.05hh}. By setting $v=\tilde N^{\epsilon'}(\cdot,x)$ in
\eqref{eq1.05hh} and setting $\psi=u$ in \eqref{eq1.36tv}, we get
\[
\int_\Omega \left( \overline{\tilde N^{\epsilon'}(z,x)} -\overline{\nabla \tilde N^{\epsilon'}(z,x)} \right) f(z) \,dz + \int_{\partial\Omega} \overline{\tilde N^{\epsilon'}(z,x)} g(z) \,d\sigma(z) = \int_ {\Omega\cap B_{\epsilon'}(x)} \Phi_{\epsilon'} u \,dz.
\]
By Lemma~\ref{lem1.1ap}, we again find that $u$ is H\"older continuous in $\Omega$.
Then by proceeding similarly as above and using \eqref{eq3.01mq}, we obtain
\[
u(x)=\int_\Omega\left( N(x,y) f(y)- \nabla_y N(x,y)\cdot F(y)\right)\,dy+\int_\Omega N(x,y) g(y)\,d\sigma(y),
\]
which is the formula \eqref{eqM1.e}.
The proof is complete. \qed
\end{proof}

%
\subsection{$L^p$ estimates}
%

We now assume that $\Omega$ is a bounded $\mathcal{C}^1$ domain.
In the following lemma we obtain $L^p$ estimates for the operator
$L$ with uniformly continuous coefficient $\gamma$.
\begin{lem}             \label{lem2.2ap}
Let $\Omega\subset \RR^d$ be a bounded $\mathcal{C}^1$ domain and
assume that $\gamma \in \mathcal{C}^0(\overline\Omega)$. Let
$q\in(1,d)$, $p\in(1,\infty)$, and $s=\min(q^*,p)$, where
$q^*=dq/(d-q)$. For each $f\in L^q(\Omega;\CC)$ and $F\in
L^p(\Omega;\CC^d)$, there is a unique weak solution $u\in
W^{1,s}(\Omega)$ to
\begin{equation}\label{eq2.00sk}
\left\{
\begin{aligned}
-L u= f +\nabla\cdot F \quad &\mbox{in }\;\Omega,\\
(\gamma \nabla u + F)\cdot n=0 \quad &\mbox{on }\;\partial\Omega.
\end{aligned}
\right.
\end{equation}
Moreover, the following estimate holds:
\begin{equation}                    \label{eq2.47rm}
\norm{u}_{W^{1,s}(\Omega)} \le C\left(\norm{f}_{L^q(\Omega)}+\norm{F}_{L^p(\Omega)}\right),
\end{equation}
where $C$ depends on $d, \nu, k_0, p, q, \Omega$, and $\theta$.
\end{lem}
\begin{proof}
Note that in the case when $f\equiv 0$, the proof for estimate
\eqref{eq2.47rm} reduces to
\begin{equation}\label{eq2.48fc}
\norm{u}_{W^{1,p}(\Omega)} \le C\norm{F}_{L^p(\Omega)}.
\end{equation}
In this case the proof for the existence and uniqueness of weak
solution $u\in W^{1,p}(\Omega)$ as well as the estimate
\eqnref{eq2.48fc} follow essentially from  the same argument as in
\cite{Krylov08}.

We consider the case when $f$ is not identically zero.
Observe that $L^q(\Omega) \subset W^{-1,q^*}_0(\Omega)$ with the estimate
\[
\norm{f}_{W^{-1,q^*}(\Omega)} \le C \norm{f}_{L^q(\Omega)},\quad \text{where }\;C=C(d,\Omega).
\]
Then by \cite[Corollary~9.3]{FMM}, there exists a unique weak solution $v$ in $W^{1,q^*}(\Omega)$ of the Neumann problem
\[
\left\{
\begin{aligned}
\Delta v = f - \frac{1}{|\Omega|} \int_\Omega f \, dy \quad &\text{in }\;\Omega,\\
\partial v/\partial n = 0 \quad &\text{on }\;\partial\Omega,
\end{aligned}
\right.
\]
where $|\Omega|$ is the volume of $\Omega$. Moreover, $v$
satisfies the estimate
\begin{equation}            \label{eq2.50mx}
\norm{\nabla v}_{L^{q^*}(\Omega)} \le C\norm{f -\frac{1}{|\Omega|} \int_\Omega f \, dy}_{W^{-1,q^*}_0(\Omega)} \le
C \norm{f}_{L^q(\Omega)}.
\end{equation}

Then, we apply estimate \eqref{eq2.48fc} with $F+ \nabla v + (
\tfrac{1}{d |\Omega|} \int_\Omega f \,dy )  x $ and $s$ in place of
$F$ and $p$, respectively, and use H\"older's inequality to get
estimate \eqref{eq2.47rm}. 
\qed
\end{proof}

We denote by $L^{p,\infty}(\Omega)$ the usual weak $L^p$ space.
The following lemma is a variant of Lemma~\ref{lem2.2ap} in the weak Lebesgue spaces.

\begin{lem}             \label{lem2.3ap}
Let $\Omega\subset \RR^d$ be a bounded $\mathcal{C}^1$ domain and
assume that $\gamma \in \mathcal{C}^0(\overline\Omega)$. Let $q\in
(1,d)$, $p\in (1,\infty)$, and $s=\min(q^*,p)$, where
$q^*=dq/(d-q)$. If $f\in L^{q,\infty}(\Omega;\CC)$ and $F\in
L^{p,\infty}(\Omega;\CC^d)$, there is a  weak solution $u$ of
problem \eqref{eq2.00sk} that satisfies an estimate
\[
\norm{\nabla u}_{L^{s,\infty}(\Omega)} \le C\left(\norm{f}_{L^{q,\infty}(\Omega)}+\norm{F}_{L^{p,\infty}(\Omega)}\right),
\]
and, for $s<d$, the following estimate as well:
\[
\norm{u}_{L^{s^*,\infty}(\Omega)} \le C\left(\norm{f}_{L^{q,\infty}(\Omega)}+\norm{F}_{L^{p,\infty}(\Omega)}\right).
\]
Moreover, there is uniqueness of weak solutions to
\eqref{eq2.00sk} in the sense that if $\tilde u$ is a solution in
$W^{1,t}(\Omega)$ for some $t>1$, then $u=\tilde u$.
\end{lem}
\begin{proof}
The lemma follows immediately from Lemma~\ref{lem2.2ap} by applying \cite[Lemma~1]{DM} to the solution operator $T: F\mapsto u$ as well as to the map $f\mapsto v$ in \eqref{eq2.50mx}. \qed
\end{proof}

\begin{lem}\label{lem:lb}
Let $\Omega$ and $\gamma$ satisfy the same assumptions as in
Lemma~\ref{lem2.2ap}. There exists a constant $C_1>0$ such that
the following holds: For any $f \in
\mathcal{C}_c^\infty(\Omega;\CC)$, let $u \in W^{1,2}(\Omega;\CC)$
be the unique weak solution of
\[
\begin{aligned}
\left\{
\begin{aligned}
-L u&= f \;\text{ in }\;\Omega\\
 \gamma \nabla u \cdot  n&= 0 \; \text{ on }\;\partial\Omega
\end{aligned}
\right.
\qquad\;\text{ or }\;\qquad
\left\{
\begin{aligned}
-\Lt  u&= f \quad\text{in }\;\Omega\\
\gamma \nabla u \cdot  n&= 0 \quad \text{on }\;\partial\Omega.
\end{aligned}
\right.
\end{aligned}
\]
Then for all $x\in\Omega$ and $0<R < \diam(\Omega)$, we have
\[
\norm{u}_{L^\infty(\Omega\cap B_{R/2}(x))} \le C_1 \left(R^{-d/2} \norm{u}_{L^2(\Omega\cap B_R(x))}+ R^2 \norm{f}_{L^\infty(\Omega \cap B_R(x))}\right).
\]
The constant $C_1$ depends on $d, \nu, \Omega$, and $\theta$.
\end{lem}
\begin{proof}
We will only consider the case when $u$ is a weak solution of  $-Lu=f$ with zero conormal data.
By Lemma~\ref{lem2.2ap}, we find that $u \in W^{1,p}(\Omega)$ for all $p\in (1,\infty)$ and
\[
\norm{\nabla u}_{L^p(\Omega)} \le C \norm{f}_{L^{pd/(p+d)}(\Omega)} \le C \norm{f}_{L^\infty(\Omega)} .
\]
Let $v=\zeta u$, where $\zeta: \RR^d \to \RR$ is a smooth function to be chosen later.
Observe that $v$ is a weak solution of the problem
\[
\left\{
\begin{aligned}
-L v= \tilde f +\nabla\cdot \tilde F \quad &\text{in }\;\Omega,\\
(\gamma \nabla v + \tilde F)\cdot n=0 \quad &\text{on }\;\partial\Omega,
\end{aligned}
\right.
\]
where
\[
\tilde f:= \zeta f -\gamma \nabla \zeta \cdot \nabla u, \qquad \tilde F:=  -\gamma u \nabla \zeta.
\]
Let $x\in\Omega$ and $0<R<\diam(\Omega)$ be arbitrary but fixed.
For any $y\in \Omega\cap B_R(x)$ and $0<\rho<r \le R$, we choose the function $\zeta$ to be such that
\[
0\le \zeta \le 1,\quad \supp \zeta \subset B_r(y),\quad \zeta\equiv 1\,\text{ on }\, B_\rho(y),\quad\text{and}\quad\abs{\nabla \zeta} \le 2/(r-\rho).
\]
For any $p\in (1,\infty)$, we set $q=pd/(p+d)$ and apply Lemma~\ref{lem2.2ap} together with H\"older's inequality to get
\begin{equation}            \label{eq6.4he}
\norm{\nabla u}_{L^p(\Omega_\rho)} \le C\left( r^{1+d/p} \norm{f}_{L^\infty(\Omega_r)} +(r-\rho)^{-1}\norm{\nabla u}_{L^{pd/(p+d)}(\Omega_r)} + (r-\rho)^{-1} \norm{u}_{L^p(\Omega)}\right),
\end{equation}
where we use the notation $\Omega_r=\Omega_r(y)=\Omega\cap B_r(y)$.
Now, fix $p>d$, and let $m= [ d(1/2-1/p)]$,
\[
p_j= \frac{pd}{d+p j} \quad\text{and}\quad r_j=\rho+ \frac{(r-\rho)j}{m},\quad j=0,\ldots,m.
\]
Then we apply \eqref{eq6.4he} iteratively to get
\begin{multline*}
\norm{\nabla u}_{L^p(\Omega_\rho)}
\le \sum_{j=1}^m C^j\left(\frac{m}{r-\rho}\right)^{j-1} r_j^{1+d/p_{j-1}} \norm{f}_{L^\infty(\Omega_{r_j})} \\
+\sum_{j=1}^m C^j\left(\frac{m}{r-\rho}\right)^j \norm{u}_{L^{p_{j-1}}(\Omega_{r_j})}+C^m\left(\frac{m}{r-\rho}
\right)^m\norm{\nabla u}_{L^{p_m}(\Omega_{r_m})}.
\end{multline*}
Notice that $1<p_m\le 2$.
By using H\"older's inequality we then obtain
\begin{multline*}
\rho^{-d(1/2-1/p)}\norm{\nabla u}_{L^2(\Omega_\rho)} \le C\left(\frac{r}{r-\rho}\right)^{m-1} r^{1+d/p}\norm{f}_{L^\infty(\Omega_r)}\\
+C \left(\frac{r}{r-\rho}\right)^m r^{-1} \norm{u}_{L^p(\Omega_r)}+ C\left(\frac{r}{r-\rho}\right)^m r^{d(1/p-1/2)} \norm{\nabla u}_{L^2(\Omega_r)}.
\end{multline*}
If we take $r=R/4$ and $\rho<r/2=R/4$ in the above, then for all $y\in \Omega_{R/4}(x)$, we get
\begin{multline}                    \label{eq6.5tt}
\left(\rho^{-(d-2+2(1-d/p))}\int_{\Omega_\rho(y)}\abs{\nabla u}^2\,dz\right)^{1/2} \le  C R^{1+d/p}\norm{f}_{L^\infty(\Omega_R(x))}\\
+C R^{-1} \norm{u}_{L^p(\Omega_R(x))}+CR^{d(1/p-1/2)} \norm{\nabla u}_{L^2(\Omega_{R/2}(x))}=:A(R).
\end{multline}
Hereafter in the proof, we shall denote $\Omega_R=\Omega_R(x)$.
Then by Morrey-Campanato's theorem (see \cite[Section~3.1]{Gi93}), for all $z, z' \in \Omega_{R/4}$, we have
\[
\abs{u(z)-u(z')} \le C R^{1-d/p} A(R),
\]
where $A(R)$ is as defined in \eqref{eq6.5tt}.
Therefore, for any $z\in \Omega_{R/4}$ we have
\[
\abs{u(z)} \le \abs{u(z')}+ \abs{u(z)-u(z')} \le \abs{u(z')} + C R^{1-d/p} A(R),\quad \forall z'\in \Omega_{R/4}.
\]
By taking average over $z'\in \Omega_{R/4}$ in the above and using the definition of $A(R)$, we obtain
\[
\sup_{\Omega_{R/4}}\,\abs{u}
\le \frac{1}{\GO_{R/4}|} \int_{\Omega_{R/4}} \abs{u(z')}\,dz'+ C R^2\norm{f}_{L^\infty(\Omega_R)}+C R^{-d/p} \norm{u}_{L^p(\Omega_R)}+CR^{1-d/2} \norm{\nabla u}_{L^2(\Omega_{R/2})}.
\]
Then by using H\"older's inequality and Caccioppoli's inequality, we get
\[
\sup_{\Omega_{R/4}}\,\abs{u}  \le C R^2\norm{f}_{L^\infty(\Omega_R)}+C R^{-d/p} \norm{u}_{L^p(\Omega_R)}+CR^{-d/2} \norm{u}_{L^2(\Omega_R)}.
\]
By using a standard argument (see \cite[pp. 80--82]{Gi93}), we derive from the above inequality
\[
\sup_{\Omega_{R/2}}\, \abs{u} \le
CR^2\norm{f}_{L^\infty(\Omega_R)}+CR^{-d/2} \norm{u}_{L^2(\Omega_R)}.
\]
The proof is complete. \qed
\end{proof}

%
\subsection{Global estimates for Neumann function}
%

The next theorem provides global pointwise bound for the Neumann
function  $N$.

\begin{thm}\label{thm2ap}
Let $\Omega\subset \RR^d$ be a bounded $\mathcal{C}^1$ domain and
assume that $\gamma \in \mathcal{C}^0(\overline\Omega)$. Let
$N(x,y)$ be the Neumann function of $L$ in $\Omega$ as constructed
in Theorem~\ref{thm1ap}. Then we have the following global
pointwise bound for the Neumann function:
\begin{equation}                            \label{eq2.17dc}
\abs{N(x,y)} \le C \abs{x-y}^{2-d}\quad \text{for all }\, x,y\in\Omega\;\text{ with }\;x\ne y,
\end{equation}
where $C$ depends on $d, \nu, \Omega$, and  $\theta$. Moreover,
for all $y\in\Omega$ and $0<r<\diam(\Omega)$, we have
\begin{enumerate}
\item[{\rm i)}] $\norm{N(\cdot,y)}_{L^{2d/(d-2)}(\Omega\setminus
B_r(y))}+\norm{\nabla N(\cdot,y)}_{L^2(\Omega\setminus B_r(y))}
\le C r^{1-d/2}$. \item[{\rm ii)}] $\norm{N(\cdot,y)}_{L^p(\GO\cap
B_r(y))}\le C r^{2-d+d/p}\;$ for $p\in [1,\frac{d}{d-2})$.
\item[{\rm iii)}] $\abs{\set{x\in\Omega:\abs{N(x,y)}>t}}\le C
t^{-d/(d-2)}\;$ for all $t> 0$. \item[{\rm iv)}] $\norm{\nabla
N(\cdot,y)}_{L^p(\GO \cap B_r(y))}\le C r^{1-d+d/p}\;$ for $p\in
[1,\frac{d}{d-1})$. \item[{\rm v)}]
$\abs{\set{x\in\Omega:\abs{\nabla_x N(x,y)}>t}}\le C
t^{-d/(d-1)}\;$ for all $t> 0$. \item[{\rm vi)}]
$\abs{N(x,y)-N(x',y)} \le C \abs{x-x'}^{\lambda_0}
\abs{x-y}^{2-d-\lambda_0}\;$ if $\;\abs{x-x'}<\abs{x-y}/2$ for
some $\lambda_0 \in (0,1)$.
\end{enumerate}
In the above, $C$ is a constant depending on $d, \nu, k_0,
\Omega$, and $\theta$; it depends on $p$ as well in ii) and  iv).
Estimates i) -- vi) are also valid for the Neumann function
$N^*(x,y)$ of the adjoint $\Lt$.
\end{thm}
\begin{proof}
Let $y\in \Omega$ be arbitrary, but fixed. Assume that $f \in
\mathcal{C}_c^\infty(\Omega;\CC)$ is supported in
$\Omega_R(y)=\Omega\cap B_R(y)$ and let $u$ be the unique weak
solution in $W^{1,2}(\Omega;\CC)$ of problem \eqref{eq2.10yq}.
Then we have the identities \eqref{eq4.2ym} and \eqref{eq4.3wx} as
in the proof of Theorem~\ref{thm1ap}. Also, we have estimate
\eqref{eq4.4ur}, and thus by Sobolev embedding theorem, we get
\begin{equation}        \label{eq3.1a}
\norm{u}_{L^{2d/(d-2)}(\Omega)} \le C \norm{f}_{L^{2d/(d+2)}(\Omega)}\le C  R^{(2+d)/2} \norm{f}_{L^\infty(\Omega_R(y))},
\end{equation}
where $C=C(d, \nu, \Omega)$.
Then by Lemma~\ref{lem:lb} and \eqref{eq3.1a}, we obtain
\begin{equation}    \label{eq3.2v}
\norm{u}_{L^\infty(\Omega_{R/2}(y))} \le C R^2 \norm{f}_{L^\infty(\Omega_R(y))}.
\end{equation}
Hence, by \eqref{eq4.3wx} and \eqref{eq3.2v}, we conclude that
\begin{equation}\label{eq3.3y}
\Biggabs{\int_{\Omega_R(y)} N^\epsilon(x,y) \overline {f(x)} \,dx
\,} \le CR^2 \norm{f}_{L^\infty(\Omega_R(y))},\quad \forall  f\in
\mathcal{C}_c^\infty(\Omega_R(y); \CC),\;\; \forall \epsilon \in
(0,R/2).
\end{equation}
Therefore, by duality, we conclude from \eqref{eq3.3y} that
\begin{equation}    \label{eq3.3w}
\norm{N^\epsilon (\cdot,y)}_{L^1(\Omega_R(y))}\le C R^2,\quad \forall \epsilon \in (0,R/2).
\end{equation}

Next, recall that the $v=N^\epsilon(\cdot,y)$ is the unique weak
solution in $W^{1,2}(\Omega;\CC)$ of problem \eqref{eq4.0ik}. Let
$x\in\Omega$,  $r>0$, and $\epsilon>0$ be such that
$B_\epsilon(y)\cap B_r(x) =\emptyset$. Then Lemma~\ref{lem:lb}
implies that
\begin{equation}                    \label{eq4.25dd}
\norm{N^\epsilon(\cdot, y)}_{L^\infty(\Omega_{r/2}(x))} \le C  r^{-d/2} \norm {N^\epsilon(\cdot, y)}_{L^2(\Omega_r(x))}.
\end{equation}
By a standard iteration argument (see \cite[pp. 80--82]{Gi93}), we then obtain from \eqref{eq4.25dd} that
\begin{equation}    \label{eq2.8r}
\norm{N^\epsilon(\cdot, y)}_{L^\infty(\Omega_{r/2}(x))}  \le C r^{-d} \norm{N^\epsilon(\cdot, y)}_{L^1(\Omega_r(x))}.
\end{equation}
Now, for any $x\in\Omega\setminus \set{y}$, take $R=3r=3\abs{x-y}/2$.
Then by \eqref{eq2.8r} and \eqref{eq3.3w}, we obtain for all $\epsilon \in (0,r)$ that
\[
\abs{N^\epsilon(x,y)} \le C r^{-d} \norm{N^\epsilon(\cdot, y)}_{L^1(\Omega_r(x))}  \le C r^{-d} \norm{N^\epsilon(\cdot,y)}_{L^1(\Omega_{3r}(y))}
\le C \abs{x-y}^{2-d}.
\]
Therefore, by using \eqref{eq4.17ht}, we may take the limit $\epsilon\to 0$ in the above  and obtain \eqref{eq2.17dc}.

To derive estimates i) -- vi) in the theorem, we need to repeat
some steps in the proof of Theorem~\ref{thm1ap} with a little
modification. Let $v=N^\epsilon(\cdot,y)$, where
$0<\epsilon<\min(d_y,r)/6$ and $0<r<\diam(\Omega)$. Let $\eta$ be
a smooth function on $\RR^d$ satisfying the conditions
\eqref{eq4.19h}. We set $\phi=\eta^2 v$ in \eqref{eq4.1tv} and
obtain
\[
\int_{\Omega} (\gamma \eta^2 \nabla v \cdot \overline{\nabla v}+ik
v \overline v)\,dx+\int_{\Omega}2 \eta \gamma \overline v \nabla v
\cdot  \nabla \eta \,dx= 0,
\]
where we used the fact that $\eta^2 \Phi_\epsilon \equiv 0$.
By using Cauchy's inequality we get
\[
\int_\Omega \eta^2 \abs{\nabla N^\epsilon(\cdot,y)}^2\,dx
\le C \int_\Omega\abs{\nabla \eta}^2 \abs{N^\epsilon(\cdot,y)}^2\,dx.
\]
By using the pointwise bound for $N^\epsilon(x,y)$ obtained above, we get
\[
\int_{\Omega\setminus B_r(y)} \abs{\nabla N^\epsilon(\cdot,y)}^2 \,dx
\le C r^{-2} \int_{B_r(y)\setminus B_{r/2}(y)} \abs{x-y}^{4-2d}\,dx \le C r^{2-d}.
\]
By taking the limit $\epsilon\to 0$ in the above, we get
\[
\norm{\nabla N(\cdot,y)}_{L^2(\Omega\setminus B_r(y))}\le C r^{(2-d)/2},\quad0<\forall r<\diam(\Omega).
\]
Observe that the pointwise bound \eqref{eq2.17dc} together with the above estimate yields
\begin{equation}\label{eq4.26gg}
\norm{N(\cdot,y)}_{L^{2d/(d-2)}(\Omega\setminus B_r(y))}+\norm{\nabla N(\cdot,y)}_{L^2(\Omega\setminus B_r(y))}\le C r^{(2-d)/2},\quad0<\forall r<\diam(\Omega),
\end{equation}
where $C$ depends on $d ,\nu, \Omega$, and $\theta$.

By following literally the same steps used in deriving
\eqref{eq4.12eq} -- \eqref{eq4.15tw} from \eqref{eq4.11bs}, and
using the fact that $\abs{\Omega}<\infty$, we obtain estimates i)
-- v) from \eqref{eq2.17dc} and \eqref{eq4.26gg}.

Finally, we remark that the proof of Lemma~\ref{lem:lb} in fact
implies that there exist constants $\lambda_0\in (0,1]$ and
$C_1>0$, which depend on $d ,\nu, \Omega$, and  $\theta$, such
that for all $x\in\Omega$ and $0<R<\diam(\Omega)$, the following
holds: Let $u$ be a  weak solution in $W^{1,2}(\Omega_R(x))$ of
either
\begin{align*}
-L u&=0 \;\text{ in }\; \Omega \cap B_R(x),\quad \gamma \nabla u \cdot n=0 \;\text{ on }\; \partial\Omega \cap B_R(x)\\
\text{or}\quad -\Lt  u&=0 \;\text{ in }\; \Omega \cap B_R(x),\quad \gamma \nabla u \cdot  n=0 \;\text{ on }\; \partial\Omega \cap B_R(x),
\end{align*}
then we have
\[
R^{\lambda_0} [u]_{0,\lambda_0;\Omega_{R/2}} \le C_1 R^{-d/2}
\norm{u}_{L^2(\Omega_R)}.
\]
By utilizing the above estimate and modifying the proof for
estimate vii) in Theorem~\ref{thm1ap}, we have vi), and the proof
is complete. \qed
\end{proof}

\subsection{Proof of Theorem \ref{thm27}}

We are now ready to prove Theorem \ref{thm27}.

 Let
$u=N(\cdot,y)-N_0(\cdot,y)$. Observe that Theorem~\ref{thm2ap}
implies that $u\in W^{1,q}(\Omega)$ for $1\le q <d/(d-1)$, and
also that we have
\[
\int_\Omega (\gamma \nabla u \overline{\nabla \phi} + ik u
\overline \phi) \,dx = \int_\Omega (\gamma-\gamma_0) \nabla
N_0(\cdot,y) \overline{\nabla \phi}\,dx,\quad \forall \phi \in
\mathcal{C}^\infty(\overline\Omega;\CC).
\]
In other words, $u$ is a weak solution in $W^{1,q}(\Omega)$ of the problem
\[
\left\{
\begin{aligned}
-Lu=-\nabla\cdot F \quad\text{in }\;\Omega,\\
(\gamma \nabla u +F) \cdot n=0 \quad \text{on }\;\partial \Omega,
\end{aligned}
\right.
\]
where $F=(\gamma-\gamma_0)\nabla N_0(\cdot,y)$.

Note that
\begin{equation}\label{eq1.70kp}
\abs{\nabla_x N_0(x,y)} \le C \abs{x-y}^{1-d},\quad \forall x,y\in \Omega, \quad x\neq y.
\end{equation}
Indeed, for any $x\in\Omega$ with $x\ne y$, we set $R=\abs{x-y}/2$
and apply \eqref{eq3rd} and estimate i) in Theorem~\ref{thm2ap} to
obtain
\[
\abs{\nabla_x N_0(x,y)} \le C R^{-d/2} \norm{\nabla N_0(\cdot,y)}_{L^2(\Omega\setminus B_R(y))}\ \le C R^{1-d},
\]
which obviously implies \eqref{eq1.70kp}.
Moreover, by repeating the same argument, we have
\begin{equation}\label{eq3.61wy}
\abs{\nabla_x^k N_0(x,y)} \le C \abs{x-y}^{2-d-k},\quad \forall x,y\in \Omega, \quad x\neq y,\quad k=1,2,\ldots.
\end{equation}
We then obtain
\begin{equation}                    \label{eq2.60bo}
\abs{F(x)}\le C[\gamma]_{0,\lambda;\Omega}
\,\abs{x-y}^{-d/\alpha},\quad \forall x\in\Omega,\quad x\ne y,
\end{equation}
where $\alpha=d/(d-1-\lambda)$, and hence $F\in L^q(\Omega)$ for
all $q<\alpha$. It then follows from Lemma~\ref{lem2.2ap} that $u
\in W^{1,q}(\Omega)$ for all $q\in (1,\alpha)$. In fact, by
Lemma~\ref{lem2.3ap} we have
\begin{equation}                    \label{eq1.72dy}
\norm{u}_{L^{\alpha^*,\infty}(\Omega)}+\norm{\nabla u}_{L^{\alpha,\infty}(\Omega)} \le C.
\end{equation}

Let $v=\zeta u$, where $\zeta:\RR^d\to \RR$ is a smooth function to be fixed later.
Observe that $v$ is a weak solution of the problem
\[
\left\{
\begin{aligned}
-L v= \tilde f +\nabla\cdot \tilde F \quad\text{in }\;\Omega,\\
(\gamma \nabla v + \tilde F)\cdot n=0 \quad \text{on }\;\partial\Omega,
\end{aligned}
\right.
\]
where
\[
\tilde f:= -\nabla \zeta \cdot F -\gamma \nabla \zeta \cdot \nabla u, \qquad \tilde F:=  \zeta F-\gamma u \nabla \zeta.
\]
Notice that if $\zeta\equiv 0$ on a neighborhood of $y$, then we have $\tilde f \in L^q(\Omega)$ and $\tilde F\in L^{q^*}(\Omega)$ for all $q\in (1,\alpha)$.
By Lemma~\ref{lem2.2ap}, we have $v \in W^{1,q^*}(\Omega)$ and thus, we find that $u\in W^{1,q^*}_{loc}(\Omega\setminus\set{y})$.
By repeating the above argument, if necessary, we conclude that
$u \in W^{1,s}_{loc}(\Omega\setminus\set{y})$ for some $s>d$, and thus we have $u \in L^\infty_{loc}(\Omega\setminus\set{y})$.

Next, for $x\in \Omega$ with $x\neq y$, let $R=\abs{x-y}/2$.
For any $x' \in \Omega\cap B_R(x)$ and $0<\rho<r \le R$, we choose the function $\zeta$ to be such that
\[
0\le \zeta \le 1,\quad \supp \zeta \subset B_r(x'),\quad \zeta\equiv 1\,\text{ on }\, B_\rho(x'),\quad\text{and}\quad\abs{\nabla \zeta} \le 2/(r-\rho).
\]
Notice that for all $q\in (1,d)$, we have the following estimates, where we write $\Omega_\rho=\Omega_\rho(x')=\Omega\cap B_\rho(x')$ for the simplicity of notation,
\begin{align*}
\norm{\nabla u}_{L^{q^*,\infty}(\Omega_\rho)} & \le \norm{\nabla v}_{L^{q^*,\infty}(\Omega)},\\
\norm{\nabla \zeta \cdot F}_{L^{q,\infty}(\Omega)} &\le
\norm{\nabla \zeta}_{L^\infty} \norm{F}_{L^{q,\infty}(\Omega_r)} \le
 C (r-\rho)^{-1} r^{d/q}\norm{F}_{L^\infty(\Omega_r)},\\
\norm{\nabla \zeta \cdot \nabla u}_{L^{q,\infty}(\Omega)} &\le \norm{\nabla \zeta}_{L^\infty} \norm{\nabla u}_{L^{q,\infty}(\Omega_r)} \le  C (r-\rho)^{-1} \norm{\nabla u}_{L^{q,\infty}(\Omega_r)},\\
\norm{\zeta F}_{L^{q^*,\infty}(\Omega)} &\le \norm{F}_{L^{q^*,\infty}(\Omega_r)} \le C r^{d/q-1}\norm{F}_{L^\infty(\Omega_r)}, \\
\norm{u\nabla \zeta}_{L^{q^*,\infty}(\Omega)} &\le \norm{\nabla \zeta}_{L^\infty} \norm{u}_{L^{q^*,\infty}(\Omega_r)} \le C(r-\rho)^{-1}\norm{u}_{L^{q^*,\infty}(\Omega_r)}.
\end{align*}
Therefore, by Lemma~\ref{lem2.3ap} applied to $v$, we have for all $t \in (d/(d-1),\infty)$
\begin{multline}                \label{eq2.73he}
\norm{\nabla u}_{L^{t,\infty}(\Omega_\rho)} \le C\left((r-\rho)^{-1}r^{1+d/t}\norm{F}_{L^\infty(\Omega_r)}+(r-\rho)^{-1}\norm{\nabla u}_{L^{td/(t+d),\infty}(\Omega_r)}\right.\\
\left.+r^{d/t}\norm{F}_{L^\infty(\Omega_r)}+(r-\rho)^{-1}\norm{u}_{L^{t,\infty}(\Omega_r)}\right).
\end{multline}
Now, fix $s>d$ and let $m = [d(1/\alpha-1/s)]$,
\[
s_j= \frac{sd}{d+s j} \quad\text{and}\quad r_j=\rho+ \frac{(r-\rho)j}{m},\quad j=0,\ldots,m.
\]
Recall that if $E$ is a bounded set and $0<q<p<\infty$, then
\begin{equation}\label{drexel}
\norm{f}_{L^{q,\infty}(E)} \le \norm{f}_{L^q(E)} \le \sqrt{\frac{p}{p-q}}\, \abs{E}^{1/q-1/p}\, \norm{f}_{L^{p,\infty}(E)} .
\end{equation}
With the aid of \eqref{drexel}, we apply \eqref{eq2.73he} repeatedly and argue as in the proof of Lemma~\ref{lem:lb} to obtain
\begin{multline*}
\rho^{-d(1/2-1/s)}\norm{\nabla u}_{L^2(\Omega_\rho)}
\le C\left(\frac{r}{r-\rho}\right)^m r^{d/s} \norm{F}_{L^\infty(\Omega_r)}+ C \left(\frac{r}{r-\rho}\right)^{m-1} r^{d/s}\norm{F}_{L^\infty(\Omega_r)} \\
+ Cr^{-1}\left(\frac{r}{r-\rho}\right)^m \norm{u}_{L^s(\Omega_r)}+C\left(\frac{r}{r-\rho}\right)^m r^{d(1/s-1/\alpha)}\norm{\nabla u}_{L^{\alpha,\infty}(\Omega_r)}.
\end{multline*}
If we take $r=R/4$ and $\rho<r/2=R/4$ in the above, then for all $x'\in \Omega_{R/4}(x)$, we get
\begin{multline}                \label{eq1.67rm}
\left(\rho^{-(d-2+2(1-d/s))}\int_{\Omega_\rho(x')}\abs{\nabla u}^2\,dz\right)^{1/2} \le  C R^{d/s}\norm{F}_{L^\infty(\Omega_R(x))}\\
+C R^{-1} \norm{u}_{L^s(\Omega_R(x))}+CR^{d(1/s-1/\alpha)} \norm{\nabla u}_{L^{\alpha,\infty}(\Omega)},
\end{multline}
which is analogous to \eqref{eq6.5tt} in the proof of Lemma~\ref{lem:lb}.
Then by utilizing \eqref{eq1.72dy} and proceeding as in the proof of Lemma~\ref{lem:lb}, we obtain
\[
\sup_{\Omega_{R/4}}\, \abs{u} \le CR^{-d} \norm{u}_{L^1(\Omega_R)}+ CR \norm{F}_{L^\infty(\Omega_R)} + CR^{1-d/\alpha},\quad \Omega_R=\Omega_R(x).
\]
By Lemma~\ref{lem2.3ap} and \eqref{drexel} again, we get
\[
\norm{u}_{L^1(\Omega_R)} \le C R^{2+\lambda}
\norm{u}_{L^{s^*,\infty}(\Omega_R)} \le CR^{2+\lambda}.
\]
Combining the above two inequalities and using \eqref{eq2.60bo}, we get
\[
\abs{N(x,y)-N_0(x,y)}=\abs{u(x)} \le C
R^{2-d+\lambda}+CR^{1-d/\alpha} \le C \abs{x-y}^{2-d+\lambda}.
\]
This completes the proof of \eqref{eq1.46bc}.

Next, we turn to the proof of \eqref{eq1.59gg}.
Let $u=N(\cdot,y)-N_0(\cdot,y)$ as before.
Observe that $u$ satisfies
\[
-L_0 u= \nabla\cdot (F+(\gamma-\gamma_0) \nabla u)\quad\text{in
}\;\Omega.
\]
Let $x\in \Omega$ satisfy $0<\abs{x-y}<d_y/2$ and let $R=\abs{x-y}/2$ as before.
For any $x' \in B_{R/2}(x)$ and $0<r \le R/2$, let $w$ be the unique weak solution in $W^{1,2}_0(B_r(x'))$ of
the problem
\[
\left\{
\begin{aligned}
-\gamma_0 \Delta w&=-ik u+ \nabla\cdot (F + (\gamma-\gamma_0) \nabla u)\quad\text{in }\; B_r(x'),\\
w &=0 \quad \text{on }\;\partial B_r(x').
\end{aligned}
\right.
\]
Then $w$ satisfies the following identity:
\begin{equation}                \label{eq:addon1}
\int_{B_r(x')} \gamma_0 \nabla w\cdot \overline{\nabla
w}\,dz=-\int_{B_r(x')} (ik u \overline w - (F-F_r)\cdot
\overline{\nabla w} - (\gamma-\gamma_0) \nabla u \cdot
\overline{\nabla w} )\,dz,
\end{equation}
where we use the notation
\[
F_r=F_{x',r}=\frac{1}{|B_r(x')|} \int_{B_r(x')} F \,dz.
\]
Notice that by H\"older's inequality and the Sobolev inequality, we have
\begin{align*}
\Biggabs{\int_{B_r(x')} iku \overline w} &\le Ck\left(\int_{B_r(x')}\abs{u}^{2d/(d+2)}\right)^{(d+2)/2d}\left(\int_{B_r(x')}\abs{\nabla w}^2\right)^{1/2}\\
&\le Ckr^{(d+2)/2}\norm{u}_{L^\infty(B_R)} \left(\int_{B_r(x')}\abs{\nabla w}^2\right)^{1/2}.
\end{align*}
Also, by H\"older's inequality, we have
\begin{align*}
\Biggabs{\int_{B_r(x')} (F-F_r)\cdot\overline{\nabla w}} &\le \left(\int_{B_r(x')}\abs{F-F_r}^2\right)^{1/2}\left(\int_{B_r(x')}\abs{\nabla w}^2\right)^{1/2}\\
&\le [F]_{0,\lambda;B_R} r^\lambda \abs{B_r}^{1/2} \left(\int_{B_r(x')}\abs{\nabla w}^2\right)^{1/2}\\
&\le C [F]_{0,\lambda;B_R} r^{\lambda+d/2}
\left(\int_{B_r(x')}\abs{\nabla w}^2\right)^{1/2}.
\end{align*}
Similarly, we estimate
\begin{align*}
\Biggabs{\int_{B_r(x')} (\gamma-\gamma_0)\nabla u \cdot\overline{\nabla w}} &\le \norm{\gamma-\gamma_0}_{L^\infty(B_r(x'))} \left(\int_{B_r(x')}\abs{\nabla u}^2\right)^{1/2}\left(\int_{B_r(x')}\abs{\nabla w}^2\right)^{1/2}\\
&\le [\gamma]_{0,\lambda;B_R} r^\lambda
\left(\int_{B_r(x')}\abs{\nabla u}^2\right)^{1/2}
\left(\int_{B_r(x')}\abs{\nabla w}^2\right)^{1/2}.
\end{align*}
Therefore, by using Cauchy's inequalities, we derive from \eqref{eq:addon1} and the above estimates that
\begin{equation}\label{eq1.66dx}
\int_{B_r(x')} \abs{\nabla w}^2 \,dz \le Ck^2 r^{d+2}
\norm{u}_{L^\infty(B_R)}^2 +C  r^{d+2\lambda}
[F]_{0,\lambda;B_R}^2 +Cr^{2\lambda} [\gamma]_{0,\lambda;B_R}^2
\int_{B_r(x')} \abs{\nabla u}^2 \,dz,
\end{equation}
where we use abbreviation $B_R=B_R(x)$.

Notice that $v=u-w$ satisfies
\[
\Delta v=0\quad \text{in }\; B_r(x').
\]
By well-known estimates for harmonic functions (see, for
instance, \cite[p. 78]{Gi83}), we get 
\[
\int_{B_\rho(x')} \abs{\nabla v-(\nabla v)_\rho}^2\,dz \le C (\rho/r)^{d+2} \int_{B_r(x')} \abs{\nabla v-(\nabla v)_r}^2\,dz,\quad \forall \rho \in(0,r).
\]
Then by using the triangle inequality, we get  for all $0<\rho<r$ that
\begin{align*}
\int_{B_\rho(x')} \abs{\nabla u-(\nabla u)_\rho}^2\,dz &\le
2\int_{B_\rho(x')} \abs{\nabla v-(\nabla v)_\rho}^2\,dz+
2\int_{B_\rho(x')} \abs{\nabla w-(\nabla w)_\rho}^2\,dz\\
&\le C(\rho/r)^{d+2}\int_{B_\rho(x')} \abs{\nabla v-(\nabla v)_r}^2\,dz+2\int_{B_\rho(x')} \abs{\nabla w}^2\,dz\\
&\le C(\rho/r)^{d+2}\int_{B_\rho(x')} \abs{\nabla u-(\nabla u)_r}^2\,dz+C\int_{B_r(x')} \abs{\nabla w}^2\,dz,
\end{align*}
where we have used the well known fact that
\[
\inf_{c\in \mathbb R} \int_{B_r(x)} \abs{f-c}^2\,dz = \int_{B_r(x)} \abs{f-f_r}^2\,dz.
\]
By combining the above inequality and \eqref{eq1.66dx}, we get for
all $0<\rho<r$ that
\begin{multline}\label{eq2.90ej}
\int_{B_\rho(x')} \abs{\nabla u-(\nabla u)_\rho}^2\,dz \le C\left(\frac{\rho}{r}\right)^{d+2} \int_{B_r(x')} \abs{\nabla u-(\nabla u)_r}^2\,dz\\
+Ck^2 r^{d+2} \norm{u}_{L^\infty(B_R)}^2 +C  r^{d+2\lambda}
[F]_{0,\lambda;B_R}^2+C r^{2\lambda} [\gamma]_{0,\lambda;B_R}^2
\int_{B_r(x')} \abs{\nabla u}^2 \,dz.
\end{multline}
On the other hand, by setting $\epsilon=d/s$ and $\rho=r$ in \eqref{eq1.67rm}, we get
\[
\left(r^{-(d-2\epsilon)}\int_{B_r(x')}\abs{\nabla u}^2\,dz\right)^{1/2} \le  C R^\epsilon\norm{F}_{L^\infty(B_R)} +C R^{\epsilon-1}  \norm{u}_{L^\infty(B_R)}+CR^{\epsilon-d/\alpha}.
\]
Combining the above inequalities, for all $x'\in B_{R/2}(x)$ and $0<\rho<r\le R/2$, we get
\begin{multline*}
\int_{B_\rho(x')} \abs{\nabla u-(\nabla u)_\rho}^2 \le C\left(\frac{\rho}{r}\right)^{d+2} \int_{B_r(x')} \abs{\nabla u-(\nabla u)_r}^2 +Ck^2 r^{d+2} \norm{u}_{L^\infty(B_R)}^2 +C [F]_{0,\lambda;B_R}^2\, r^{d+2\lambda}\\
+C [\gamma]_{0,\lambda;B_R}^2\, r^{d+2\lambda-2\epsilon}
\left(R^\epsilon\norm{F}_{L^\infty(B_R)} + R^{\epsilon-1}
\norm{u}_{L^\infty(B_R)}+R^{\epsilon-d/\alpha}\right)^2.
\end{multline*}
By Campanato's iteration lemma, for all $x'\in B_{R/2}(x)$ and
$0<r\le R/2$, we have
\begin{multline*}
\int_{B_r(x')} \abs{\nabla u-(\nabla u)_r}^2 \le C\left(\frac{r}{R}\right)^{d+2\beta} \int_{B_R} \abs{\nabla u}^2 +Ck^2 r^{d+2\beta}R^{2-2\beta} \norm{u}_{L^\infty(B_R)}^2 +C [F]_{0,\lambda;B_R}^2\,R^{2\lambda-2\beta} r^{d+2\beta}\\
+C [\gamma]_{0,\lambda;B_R}^2\, r^{d+2\beta}
R^{2\lambda-2\beta}\left(\norm{F}_{L^\infty(B_R)} + R^{-1}
\norm{u}_{L^\infty(B_R)}+R^{-d/\alpha}\right)^2,
\end{multline*}
where we set $\beta:=\lambda-\epsilon\in (0,1)$. Therefore, by
Campanato's theorem, we obtain
\begin{multline*}
R^\beta [\nabla u]_{0,\beta;B_{R/2}} \le C R^{-d/2} \norm{\nabla u}_{L^2(B_R)} +C k R \norm{u}_{L^\infty(B_R)}+ CR^\lambda [F]_{0,\lambda;B_R}\\
+C [\gamma]_{0,\lambda;B_R} R^\lambda
\left(\norm{F}_{L^\infty(B_R)} + R^{-1}
\norm{u}_{L^\infty(B_R)}+R^{-d/\alpha}\right).
\end{multline*}
By Caccioppoli's inequality, we estimate
\[
\norm{\nabla u}_{L^2(B_R)} \le C R^{-1} \norm{u}_{L^2(B_{3R/2})}
+C \norm{F}_{L^2(B_{3R/2})} \le CR^{d/2} R^{1-d+\lambda}.
\]
Also, observe that
\[
[F]_{0,\lambda;B_R} \le R^\lambda [\gamma]_{0,\lambda;B_R} [\nabla
N_0(\cdot,y)]_{0,\lambda;B_R} + [\gamma]_{0,\lambda;B_R}
\norm{\nabla N_0(\cdot,y)}_{L^\infty(B_R)} \le C R^{1-d},
\]
where $C$ depends on $d, \nu, \lambda, \Omega$, and
$[\gamma]_{0,\lambda;\Omega}$. Therefore,
\begin{equation}                \label{eq1.69yy}
R^\beta [\nabla u]_{0,\beta;B_{R/2}} \le C\left( R^{1-d+\lambda}
+kR^{3-d+\lambda} + R^{1-d+2\lambda} \right) \le CR^{1-d+\lambda}
(1+kR^2),
\end{equation}
where we used the assumption that $\Omega$ is bounded in the last step.
By proceeding as in the proof of Lemma~\ref{lem:lb}, we derive from \eqref{eq1.69yy} that
\[
\sup_{B_{R/4}} \,\abs{\nabla u} \le CR^{1-d+\lambda} (1+k R^2).
\]
This completes the proof of \eqref{eq1.59gg}.

Now, let us assume that $\gamma\in
\mathcal{C}^{1,\lambda}(\Omega)$. Let $x\in \Omega$ satisfy
$0<\abs{x-y}<d_y/2$. We again set $R=\abs{x-y}/2$ and write
$B_R=B_R(x)$. Observe that  $u$ satisfies
\[
-\gamma \Delta u = f \quad\text{in }\;B_R,
\]
where
\[
f := \nabla \gamma \cdot \nabla u - iku + ik \gamma_0^{-1}(\gamma-\gamma_0) N_0(\cdot,y) + \nabla \gamma \cdot \nabla N_0(\cdot,y).
\]
We claim that $f\in \mathcal{C}^{0,\lambda}(\overline B_R)$.
Indeed, observe that by feeding estimate \eqref{eq1.59gg} back to
\eqref{eq2.90ej} and repeating the above steps, we obtain an
improved version of estimate \eqref{eq1.69yy}, namely,
\[
[\nabla u]_{0,\lambda;B_{R/2}} \le CR^{1-d}(1+kR^2).
\]
Therefore, we obtain
\begin{align*}
[\nabla\gamma\cdot \nabla u]_{0,\lambda;B_R}
&\le [\nabla \gamma]_{0,\lambda;\Omega}\norm{\nabla u}_{L^\infty(B_R)}+\norm{\nabla\gamma}_{L^\infty(\Omega)}[\nabla u]_{0,\lambda;B_R}\\
&\le CR^{1-d+\lambda}(1+kR^2)+CR^{1-d}(1+kR^2) \le C
R^{1-d}(1+kR^2),
\end{align*}
where we have used the assumption that $\Omega$ is bounded. Also,
by taking $s=d/(1-\lambda)$ in \eqref{eq1.67rm}, we find that for
all $x'\in B_{R/2}$ and $\rho\le R/4$, we have
\[
\left(\rho^{-(d-2+2\lambda)}\int_{B_\rho(x')}\abs{\nabla
u}^2\,dz\right)^{1/2} \le  C R^{1-\lambda}\norm{F}_{L^\infty(B_R)}
+C R^{-\lambda} \norm{u}_{L^\infty(B_R)}+CR^{2-d} \le C R^{2-d}.
\]
From the above inequality and a standard covering argument, we
find that
\[
[iku]_{0,\lambda;B_R}=k[u]_{0,\lambda;B_R} \le C k R^{2-d}.
\]
In a similar fashion, with the aid of \eqref{eq3.61wy}, we also
estimate
\begin{align*}
[ik \gamma_0^{-1}(\gamma-\gamma_0) N_0(\cdot,y)]_{0,\lambda;B_R}
\le CkR^{2-d},\\
[\nabla\gamma\cdot \nabla N_0(\cdot,y)]_{0,\lambda;B_R}
\le CR^{1-d-\lambda}.
\end{align*}
Combining all together, we find
\[
[f]_{0,\lambda;B_R} \le
CR^{1-d}(1+kR^2)+Ck R^{2-d}+CR^{1-d-\lambda}
\le CR^{-d}(1+kR^2),
\]
where we again used that $\diam \Omega<\infty$.
Then the interior Schauder estimate yields
\[
[\nabla^2 u]_{0,\lambda;B_{R/2}} \le
C\left([f]_{0,\lambda;B_R}+R^{-2-\lambda}\norm{u}_{L^\infty(B_R)}\right)
\le CR^{-d} (1+ k R^2).
\]
On the other hand, by the standard $L^2$ estimates, we have
\begin{align*}
\norm{\nabla^2 u}_{L^2(B_{R/2})} &\le C\left(R^{-1}\norm{\nabla u}_{L^2(B_R)}+\norm{\nabla u}_{L^2(B_R)} +\norm{\nabla F}_{L^2(B_R)}\right)\\
&\le C\left((R^{-2}+R^{-1}) \norm{u}_{L^2(B_{3R/2})}+(R^{-1}+1)\norm{F}_{L^2(B_{3R/2})}+\norm{\nabla F}_{L^2(B_R)} \right)\\
&\le CR^{d/2} R^{-d+\lambda}(1+R^{1-\lambda}+R) \le
CR^{d/2}R^{-d+\lambda}.
\end{align*}
Therefore, we have
\begin{align*}
\sup_{B_{R/4}}\;\abs{\nabla^2 u} &\le C R^{-d/2} \norm{\nabla^2 u}_{L^2(B_{R/2})}+ CR^\lambda [\nabla^2 u]_{0,\lambda;B_{R/2}} \\
&\le CR^{-d+\lambda}+CR^{-d+\lambda} (1+ k R^2) \le C
R^{-d+\lambda} (1+ k R^2).
\end{align*}
We have thus proved \eqref{eq2.60gg}.
Finally, we prove \eqref{eq2.61hj} as follows.
Notice that $v:=\partial u/\partial x_i$, for $i=1,\ldots, d$, satisfies
\[
-Lv= \nabla\cdot \tilde{F},\quad \text{where }\; \tilde F=
(\partial\gamma/\partial x_i)\nabla u+\partial F/\partial x_i.
\]
Let $R=\abs{x-y}/2$ as before and applying \eqref{eq1.14vt} in Lemma~\ref{lem1.1ap} to $v$, we obtain
\[
\sup_{B_{R/2}}\;\abs{v} \le C\left( R^{-d/2} \norm{\nabla u}_{L^2(B_R)}+R^{1-d/2} \norm{\tilde F}_{L^2(B_R)}\right).
\]
Notice that
\begin{align*}
\norm{\tilde F}_{L^2(B_R)} &\le C \norm{\nabla\gamma}_{L^\infty(\Omega)} \left(\norm{\nabla u}_{L^2(B_R)}+ \norm{\nabla N_0(\cdot,y)}_{L^2(B_R)}+R\norm{\nabla^2 N_0(\cdot,y)}_{L^2(B_R)}\right)\\
&\le C \norm{\nabla u}_{L^2(B_R)}+ C R^{d/2}R^{1-d}.
\end{align*}
On the other hand, by \eqref{eq1.67rm}, we find
\[
R^{-d/2}\norm{\nabla u}_{L^2(B_R)} \le C
\left(\norm{F}_{L^\infty(B_{3R/2})}+R^{-1}\norm{u}_{L^\infty(B_{3R/2})}+
R^{-d/\alpha}\right)\le CR^{1-d+\lambda}.
\]
Combining together, we obtain
\[
\sup_{B_{R/2}}\;\abs{\nabla u} \le C\left( R^{1-d+\lambda}
+R^{2-d+\lambda} +R^{2-d}\right)\le C R^{1-d+\lambda},
\]
where $C$ depends on $\norm{\nabla \gamma}_{L^\infty(\Omega)}$ and
$\diam \Omega$ as well as on $d,\nu,k_0,\lambda, \Omega$. This
proves estimate \eqref{eq2.61hj}. The proof of Theorem \ref{thm27}
is now complete.

\begin{rem}     \label{rmk2.8}
We remark that for $z\in\Omega$ fixed, we may choose
$\epsilon=\epsilon(z)>0$ so small that for all $y\in
B_\epsilon(z)$, we have $d_y>4\epsilon$. Then all $x,y \in
B_\epsilon(z)$ should satisfy the relation $\abs{x-y}<d_y/2$, and
hence \eqnref{eq1.46bc}-\eqnref{eq2.60gg} hold for all $x,y \in
B_\epsilon(z)$. We also note that in the proof of
\eqref{eq2.61hj}, it is enough to assume that $\gamma \in
\mathcal{C}^1(\overline\Omega)$ not $\gamma\in
\mathcal{C}^{1,\lambda}(\overline\Omega)$. Also, if we assume
$\gamma\in \mathcal{C}^2(\overline\Omega)$, then instead of
\eqref{eq2.60gg}, we have
\begin{equation}                \label{eq3.71pb}
\abs{\nabla^2_x N(x,y)-\nabla^2_x N_0(x,y)} \le C
\abs{x-y}^{-d+\lambda},\quad\forall x\in\Omega\;\text{ satisfying
}\;0<\abs{x-y}<d_y/2,
\end{equation}
where $C$ depends on $\norm{\gamma}_{\mathcal{C}^2(\Omega)}$ and
$\diam\Omega$ as well as on $d,\nu,k_0,\lambda, \Omega$. The proof
for \eqref{eq3.71pb} is analogous to that for \eqref{eq2.61hj}.
Moreover, if $d\ge 4$, then we may take $\gamma=1$ in
\eqref{eq1.46bc}, \eqref{eq2.61hj}, and \eqref{eq3.71pb} since in
that case, we may take $\alpha=d/(d-2)<d$ in \eqref{eq1.72dy}.
\end{rem}

%
\section{Applications to quantitative photo-acoustic imaging} \label{sec:app}
%

In this section we deal with the problem of quantitative
photo-acoustic imaging to reconstruct the optical absorption
coefficient from the absorbed energy density. The absorbed energy
density can be reconstructed using the measurements of the
acoustic wave on the boundary of the medium. See, for instance,
\cite{ABJK, wang}.

Reconstruction of the optical absorption coefficient, $\mu_a$,
from the absorbed energy density, $A(x)$, is subtle since $\mu_a$
is related to $A(x)$ in a nonlinear and implicit way. In fact,
$\mu_a$ is related to $A(x)$ by
 \beq\label{AmuPhi}
 A = \mu_a \Phi
 \eeq
Here $\Phi$ is the light fluence which depends on the distribution
of scattering and absorption within $\Omega$, as well as the light
sources. Let $\mu_s$ be the scattering coefficient. The function
$\Phi$ is related to $\mu_a$ through the diffusion equation
 \beq \label{phi}
 \bigg( \frac{i \omega}{c} + \mu_a(x) - \frac{1}{3} \nabla \cdot \frac{1}{\mu_a(x) + \mu_s(x)} \nabla \bigg) \Phi (x) = 0 \quad \mbox{in } \GO,
 \eeq
with the boundary condition
 \beq\label{phib}
 \frac{1}{3(\mu_a(x) + \mu_s(x))}\frac{\partial \Phi}{\partial \nu} = g \quad \mbox{on } \partial \Omega,
 \eeq
where $g$ denotes the light source and $\omega$ is a given
frequency. Equation \eqnref{phi} is derived based on the diffusion
approximation to the transport equation which holds when $\mu_s
\gg \mu_a$. See, for instance,  \cite{arridge, kaipio}. Note that
in \cite{ABJK2}, the boundary condition is a Robin boundary
condition. However, it is easy to check that all the estimates
derived in  \cite[Section 2]{ABJK2} hold for the Neumann boundary
condition \eqnref{phib}.

We restrict ourselves to the three-dimensional case and suppose
that the medium contains a small absorbing anomaly whose
absorption coefficient is to be reconstructed. The small unknown
anomaly $D$ is modeled as
 \beq\label{DzB}
 D = z + \Ge B,
 \eeq
where $z$ represents the location of $D$, $B$ is a reference domain which contains the origin, and $\Ge$ is a small parameter representing the diameter of the anomaly. We assume that the anomaly is away from the
boundary $\p \GO$, namely
 \beq \label{cone}
 \mbox{dist}(z, \partial \Omega) \ge C_0
 \eeq
for some constant $C_0$. Since $D$ is small and absorbing, and the background absorption is quite small compared to the scattering, we may assume that
 \beq\label{GmaD}
 \Gm_a(x) = \Gm_a \chi_D(x)
 \eeq
where $\Gm_a$ is a constant and $\chi_D$ is the characteristic
function of $D$. Then, \eqnref{phi} and \eqnref{phib} may be
approximated by \beq \label{Phiphoto} \left\{
\begin{array}{rl}
\ds  \bigg( \frac{i \omega}{c} + \mu_a \chi_D(x) - \frac{1}{3} \nabla \cdot \frac{1}{\mu_s(x)} \nabla \bigg) \Phi (x) &= 0 \quad \mbox{in } \GO, \\
\nm \ds \frac{1}{3 \mu_s(x)}\frac{\p \Phi}{\p \nu} &= g \quad
\mbox{on } \p \GO.
 \end{array} \right.
 \eeq
Since $D$ is small, we may regard $\Phi$ as a perturbation of
$\Phi^{(0)}$ which is the solution of
 \beq \label{Phizero}
 \left\{
 \begin{array}{rl}
 \ds  \bigg( \frac{i \omega}{c} - \frac{1}{3} \nabla \cdot \frac{1}{\mu_s(x)} \nabla \bigg) \Phi^{(0)} (x) &= 0 \quad \mbox{in } \GO, \\
 \nm \ds \frac{1}{3 \mu_s(x)}\frac{\p \Phi^{(0)}}{\p \nu} &= g \quad \mbox{on } \p \GO.
 \end{array} \right.
 \eeq

The reconstruction methods in \cite{ABJK2} deeply rely on the following asymptotic formula ${\Phi} -\Phi^{(0)}$, which was obtained under
the assumption that $\mu_s$ is constant:
\begin{equation} \label{eqasymp}
({\Phi} -\Phi^{(0)})(z) \approx  3 \Ge^2 \mu_a \mu_s \Phi^{(0)}(z)
\hat{N}_B(0) - \Ge \frac{\mu_a}{ \mu_s} \Scal_B[\nu](0) \cdot
\nabla \Phi^{(0)}(z),
\end{equation}
where $\hat{N}_B$ be the Newtonian potential of $B$, which is given by
 \beq
 \hat{N}_B(x): = \int_B
 \Gamma(x - y)\, dy, \quad x \in \RR^3,
 \eeq
and $\Scal_B$ is the single layer potential associated to $B$,
which is given for a density $\psi \in L^2(\partial B)$ by
 $$
 \Scal_B [\psi] (x) := \int_{\p B} \Gamma (x-y) \psi (y) \, d
 \sigma(y) , \quad x \in \RR^3.
 $$
Here $\Gamma$ is the fundamental solution to the Laplacian in
three dimensions, {\it i.e.},
 $$
 \Gamma(x):= - \frac{1}{4\pi |x|}.
 $$

The purpose of this section is to show that the asymptotic
expansion \eqnref{eqasymp} holds even when $\mu_s$ is variable.
The following theorem holds.
\begin{thm} Let $\Omega$ be a bounded $\mathcal{C}^1$-domain in $\RR^3$.
Let $D= z + \Ge B,$ where $B$ is a bounded Lipschitz domain in
$\RR^3$ containing the origin. Suppose that $\mu_a$ is given by
\eqnref{GmaD} and $\mu_s \in
\mathcal{C}^{1,\lambda}(\overline{\Omega})$, $\lambda \in (0,1)$,
and set $\bar\mu_s:= \mu_s(z)$. Then, as $\Ge \rightarrow 0$, we
have
\begin{equation} \label{eqasymp-g}
({\Phi} -\Phi^{(0)})(z) \approx  3 \Ge^2 \mu_a \bar\mu_s
\Phi^{(0)}(z) \hat{N}_B(0) - \Ge \frac{\mu_a}{\bar\mu_s}
\Scal_B[\nu](0) \cdot \nabla \Phi^{(0)}(z),
\end{equation}
where the error term is less than
$$
C \bigg( \Ge^{1+ {1}/{p}} \mu_a \bar\mu_s^{3/2} (1+ \Ge
\sqrt{\bar\mu_s}) \bigg( \Ge^2 \mu_a \bar\mu_s +
 \frac{\mu_a}{\bar\mu_s}\bigg) + \Ge \sqrt{\bar\mu_s} \bigg( \Ge^2 \mu_a \bar\mu_s + (\frac{\mu_a}{\bar\mu_s})^2
 \bigg)  + \Ge^2 \mu_a\bigg),
$$
for $p>3$ and some constant $C$ depending on $||
\mu_s||_{\mathcal{C}^{1,\lambda}}, \lambda, \Omega$,  $\omega/c$,
and $g$.
\end{thm}

Since the proof is essentially the same as that in \cite{ABJK2},
we only outline the proof without much details.

Let $N(x,y)$ be the Neumann function of the operator $\frac{i \omega}{c} - \frac{1}{3} \nabla \cdot \frac{1}{\mu_s(x)} \nabla$ on $\GO$. Then one can show by following the same lines of the proof of \cite[Lemma 2.1]{ABJK2} that
for any $x\in \Omega$,
 \beq\label{mainasym}
 \begin{array}{lll} \ds ({\Phi} -\Phi^{(0)})(x) &=&\ds
 \mu_a \int_D \Phi(y) N(x,y)\; dy\\ \nm && + \, \ds \frac{1}{3}
 \int_D (\frac{1}{\mu_a +\mu_s(y)} - \frac{1}{\mu_s(y)})  \nabla
 \Phi(y)\cdot \nabla_y N(x,y)\, dy. \end{array}
 \eeq
Let $N_0(x,y)$ be the Neumann function of $\frac{i \omega}{c} -
\frac{1}{3\bar\mu_s} \GD$ on $\GO$. We suppose that
$\frac{\bar\mu_s}{3\mu_s}$ satisfy the ellipticity condition
\eqnref{eqP-02}. It follows from Theorem \ref{thm27} that there is
a constant $C$ such that
\begin{align*}
\abs{N(x,y)-N_0(x,y)} & \le C \bar\mu_s \abs{x-y}^{-1+\lambda}, \\
\abs{\nabla_x (N(x,y)- N_0(x,y))} &\le C \bar\mu_s \abs{x-y}^{-2+\lambda},\\
\abs{\nabla^2_x (N(x,y)-N_0(x,y))} &\le C \left(\bar\mu_s
\abs{x-y}^{-3+\lambda} + \bar\mu_s^2 \abs{x-y}^{-1+\lambda}
\right),
\end{align*}
for all $x, y \in D$ provided that $\Go$ is bounded. On the other hand, it is proved in \cite[Lemma 2.2]{ABJK2} that there is a constant $C$ such that
 \begin{align*}
 | N_0(x,y)-3\bar\mu_s \Gamma(x-y) | & \le C \bar\mu_s^{3/2} , \\
 | \nabla_x (N_0(x,y)-3\bar\mu_s \Gamma(x-y)) | & \le C\left( \bar\mu_s^{2} + \bar\mu_s^{3/2}|x-y|^{-1} \right), \\
 | \nabla^2_x (N_0(x,y)-3\bar\mu_s \Gamma(x-y)) | & \le C \left( \bar\mu_s^{5/2} + \bar\mu_s^{3/2}|x-y|^{-2} \right),
 \end{align*}
for all $x, y \in D$ provided that $\Ge \sqrt{\bar\mu_s}$ is sufficiently small. Therefore, if we put
 \beq \label{rxy}
 R(x,y)= N(x,y)-3\bar\mu_s \Gamma(x-y),
 \eeq
we obtain the following lemma. \begin{lem} Let $R$ be defined by
\eqnref{rxy}. There exists a constant $C$ such that
 \begin{align}
 | R(x,y) | & \le C \left( \bar\mu_s^{3/2} + \bar\mu_s \abs{x-y}^{-1+\lambda} \right), \label{rxy1} \\
 | \nabla_x R(x,y) | & \le C\left( \bar\mu_s^{2} + \bar\mu_s^{3/2}|x-y|^{-1} + \bar\mu_s \abs{x-y}^{-2+\lambda} \right), \label{rxy2} \\
 | \nabla^2_x R(x,y) | & \le C \left( \bar\mu_s^{5/2} + \bar\mu_s^{3/2}|x-y|^{-2} + \bar\mu_s \abs{x-y}^{-3+\lambda} + \bar\mu_s^2 \abs{x-y}^{-1+\lambda} \right).
 \label{rxy3}
 \end{align}
\end{lem}

We introduce some notation following \cite{ABJK2}. Let
 \beq
 n(x) := \int_D N(x,y)\; dy, \quad x \in D,
 \eeq
and define a multiplier $\Mcal$ by
 \beq
 \Mcal[f](x) := \mu_a n(x) f(x).
 \eeq
We then define two operators $\Ncal$ and $\Rcal$ by
 \begin{align}
 \Ncal[f](x) &:= 3\mu_a \bar\mu_s \int_D (f(y)-f(x)) \Gamma(x-y)\; dy \nonumber \\
 & \quad \quad + \bar\mu_s \int_D (\frac{1}{\mu_a +\mu_s} - \frac{1}{\mu_s}) \nabla
 f(y)\cdot \nabla_y \Gamma(x-y) \, dy , \label{nfx} \\
 \Rcal[f](x) &:= \mu_a \int_D (f(y)-f(x)) R(x,y)\; dy \nonumber \\
 & \quad \quad + \frac{1}{3} \int_D (\frac{1}{\mu_a +\mu_s} - \frac{1}{\mu_s}) \nabla
 f(y)\cdot \nabla_y R(x,y)\, dy. \label{rfx}
 \end{align}
Then, \eqnref{mainasym} can be rewritten as
 \beq\label{inteqn}
 (I-\Mcal)[\Phi] - (\Ncal+\Rcal)[\Phi] = \Phi^{(0)} \quad \mbox{on } D,
 \eeq
where $I$ is the identity operator.

For $\eta>0$, define
 $$
 T_\eta [f](x) = \int_D \frac{f(y)}{|x-y|^{3-\eta}} dy, \quad x \in D.
 $$
Then one can show using H\"older's inequality that
 \beq\label{talpha2}
 \| T_\eta [f] \|_{L^p(D)} \le C \Ge^{\eta} \| f \|_{L^p(D)}
 \eeq
for all $p > \frac{3}{\eta}$.

We fix $\lambda$ so that $\lambda > \frac{1}{2}$. Suppose that
$\Ge \sqrt{\bar\mu_s}$ and $\frac{\mu_a}{\bar\mu_s}$. Using
\eqnref{talpha2} one can show that
 \beq
 \| \Ncal [f] \|_{W^{1,p}(D)} \le C \left(\Ge^2 \mu_a \bar\mu_s + \frac{\mu_a}{\bar\mu_s} \right)
  \left\| \nabla f \right \|_{L^{p}(D)}. \label{wonepest4p}
 \eeq
One can also show using \eqnref{rxy1}-\eqnref{rxy3} and \eqnref{talpha2} that
 \beq
  \| \Rcal [f] \|_{W^{1,p}(D)} \le C \Ge \sqrt{\mu_s} \left( \mu_a \mu_s \Ge^{2}
 + \frac{\mu_a}{\mu_s} \right) \| \nabla f \|_{L^{p}(D)}. \label{wonepest4}
 \eeq
Therefore, the following lemma holds.
\begin{lem}
Let $p > 3$. Then there exists a constant $C$ such that
\eqnref{wonepest4p} and \eqnref{wonepest4} hold.
\end{lem}

The rest of derivation of \eqnref{eqasymp} is exactly the same as in \cite{ABJK2}. But we briefly describe it for the readers' sake. From \eqnref{inteqn}, we get
 \beq
 \Phi = \sum_{j=0}^\infty \left((I-\Mcal)^{-1}(\Ncal+\Rcal) \right)^j (I-\Mcal)^{-1}
 [\Phi^{(0)}],
 \eeq
which converges in $W^{1,p}(D)$ for all $p>3$. We then obtain
 \beq\label{wonepasym}
 \Phi(x) = (I-\Mcal)^{-1} [\Phi^{(0)}](x) + (\Ncal+\Rcal) (I-\Mcal)^{-1} [\Phi^{(0)}](x) + E(x), \quad x \in
 D,
 \eeq
where the error term $E$ satisfies
 \beq\label{errorest}
 \| E \|_{W^{1,p}(D)} \le C \Ge \bar\mu_s \mu_a (1+ \Ge \sqrt{\bar\mu_s}) \left( \Ge^2 \mu_a \bar\mu_s +
 \frac{\mu_a}{\bar\mu_s} \right) \| \Phi^{(0)} \|_{W^{1,p}(D)}.
 \eeq
Then \eqnref{eqasymp} follows from \eqnref{wonepasym} and the error of the approximation is bounded by a constant times
 $$
 \Ge^{1+ {1}/{p}} \mu_a \bar\mu_s^{3/2} (1+ \Ge \sqrt{\bar\mu_s}) \bigg( \Ge^2 \mu_a \bar\mu_s +
 \frac{\mu_a}{\bar\mu_s}\bigg) + \Ge \sqrt{\bar\mu_s} \bigg( \Ge^2 \mu_a \bar\mu_s + (\frac{\mu_a}{\bar\mu_s})^2
 \bigg)  + \Ge^2 \mu_a.
 $$
We emphasize that approximation \eqnref{eqasymp} is valid under
the assumption that $\Ge \sqrt{\bar\mu_s}$ and
$\frac{\mu_a}{\bar\mu_s}$ are small, which indicates that the size
and absorption coefficient of the anomaly are much smaller than
the scattering coefficient.

\end{document}